\documentclass{article}

\usepackage{graphicx}
\usepackage[T1]{fontenc}
\usepackage[latin1]{inputenc}
\usepackage{array}
\usepackage{color}

\usepackage{amsmath,amsfonts,amssymb}
\usepackage{amsthm}
\usepackage{amsfonts}
\usepackage{stmaryrd}
\usepackage{ulem}

\let\ds=\displaystyle

\let\to=\rightarrow
\newcommand{\be}{\begin{enumerate}}
\newcommand{\ee}{\end{enumerate}}
\newcommand{\bi}{\begin{itemize}}
\newcommand{\ei}{\end{itemize}}

\def \abs#1{{\left|#1\right|}}
\def \pa#1{{\left(#1\right)}}

\def\bs{\bigskip}
\def\ms{\medskip}

\font\ineg=msam8
\def\ie{\mathrel{\hbox{\ineg 6}}} % <= français
\def\se{\mathrel{\hbox{\ineg >}}} % >= français

\font\matcinq=msbm5
\font\matsept=msbm7
\font\matdix=msbm10
\newfam\matfam \scriptscriptfont\matfam=\matcinq \scriptfont\matfam=\matsept \textfont\matfam=\matdix
\def\mat{\fam\matfam}
\def\N{{\mat N}} \def\Z{{\mat Z}}
 
\def\R{{\mat R}} \def\C{{\mat C}}

\newcommand{\bib}[2]{\hbox{\hbox to 12mm{[#1] :\hfill} \hfill \hbox to 144mm{\vtop{\hsize=144mm#2\vfill\ms}\hfill}}}

\let\phi=\varphi

\def\sin{{\rm sin}}
\def\cos{{\rm cos}}

\newtheorem{Thm}{Theorem I-\!}[section]
\newtheorem{Lem}{Lemma I-\!}[section]
\newtheorem{Cor}{Corollary I-\!}[section]

\def\twp{\tilde{\wp}} %ancienne définition
\def\wpa{\varcurlyvee} %c'est \twp à une constante près
\def\twpa{\widetilde{\varcurlyvee}} %c'est \wp à une constante près

\begin{document}

\begin{center}
\Large{\textsc{Structure and bases of modular space sequences}}

\Large{\textsc{$(M_{2k}(\Gamma_0(N)))_{k\in \mathbb{N}^*}$ and $(S_{2k}(\Gamma_0(N)))_{k\in \mathbb{N}^*}$}}
\bs
\ms

\large\textsc{Part II: a modular butterfly hunt}
\end{center}
\ms

%\author{Jean-Christophe Feauveau}
\begin{center}
\large Jean-Christophe Feauveau
\footnote{Jean-Christophe Feauveau,\\
Professeur en classes préparatoires au lycée Bellevue,\\
135, route de Narbonne BP. 44370, 31031 Toulouse Cedex 4, France,\\
email: Jean-Christophe.Feauveau@ac-toulouse.fr
}

\end{center}

\begin{center}
\large September 01, 2018
\end{center}

\bs
%%%%%%%%%%%%%%%%%%%%%%%%%%%%%%%%%%%%
%%%%%%%%%%%%% 	   	%%%%%%%%%%%%%%%%%%
%%%%%%%%%%%%%  	Abstract  	%%%%%%%%%%%%
%%%%%%%%%%%%% 	    	%%%%%%%%%%%%%%%%%%
%%%%%%%%%%%%%%%%%%%%%%%%%%%%%%%%%%%%

\textsc{Abstract.}

In the first part of this article, which contains three of them, we have identified the notion of level $N$ strong modular unit. It enabled us to structure the modular forms family $(M_{2k}(\Gamma_0(N)))_{k\in \; \mathbb{N}^*}$ and to propose the explicit bases for these spaces. It is in this perspective that we wrote this second part where the structure and explicit bases are proposed when $1\leq N \leq 10$.
\bs
\bs

\textsc{Key words.} modular forms, modular units, elliptic functions, Dedekind's eta function.

\bs

Classification A.M.S. 2010: 11F11, 11G16, 11F33, 33E05.
\bs
\bs

{\bf Introduction}
\ms

We propose in this article to describe the complete modular space structure $(M_{2k}(\Gamma_0(N)))_{k\in\N^*}$ for $N\in \llbracket 1,10\rrbracket$. This set of integers contains the first prime numbers, the square and the cube of prime numbers, as well as the product of two distinct prime numbers. Enough to adapt to other situations if necessary.
\ms

It is a question of applying concretely the results developed in the first part of this article \cite{FeauFM1}. For a given level $N$, we established the existence of a function $\Delta_N$ which enables one to structure the family $(M_{2k}(\Gamma_0(N)))_{k\in\N^*}$ and to describe bases of each of its spaces when we know a basis of each space $M_{2k}(\Gamma_0(N))$, for weights between $2$ and $\rho_N+$2, where $\rho_N$ is the weight of $\Delta_N$.
\ms

The modular forms constituting the bases for $2 \ie 2k \ie \rho_N+2$, will be described in the form of $\eta$-products, as far as $\Delta_N$ is concerned, but especially of classic or renormalized Weierstrass elliptic functions described in \cite{FeauE}. These will enable one representations on various forms of modular forms.
\ms

\section{-- Weierstrass functions on the reduced lattice $\Lambda_{\tau} = \Z + \tau\Z$}

The following results are reminders from \cite{FeauE}.

For $\tau$ in Poincaré half plane ${\cal H}$, let $\Lambda_\tau = \Z + \tau\Z$ be the lattice of periods $1$ and $\tau$.

We remind results for the functions $\wp(z,\tau)$ and $\twp(z,\tau)$, renormalized by a multiplicative factor that avoids cluttering the modular forms with unnecessary factors

\begin{equation*}
\wpa = \frac{1}{\pi^2}\wp \ \ \text{and} \ \ \twpa = \frac{1}{\pi^2}\twp.
\end{equation*}

These two functions enable one to obtain modular forms according to groups $\Gamma_0(N)$. As we will see, the $\wpa$ function is more flexible to use, whereas $\twpa$ is particularly adapted to $N = 2^n$ case and enables one to naturally obtain modular units.
\bs

\textsc{$\bullet$ The function ${\wpa}$}
\be
\item[-] \textit{Definition on series form}

For $(z,\tau)\in \pa{\C-\Lambda_\tau}\times {\cal H}$:
\begin{equation}
{\wpa}(z,\tau) = \frac{1}{\pi^2}\pa{1 + \sum_{n,m}^* \frac{1}{(z+n+m\tau)^2} - \frac{1}{(n+m\tau)^2}}, \ \ \forall z\in \C-\Lambda_\tau. \label{Wp}
\end{equation}
where it was noted $\ds \sum_{n,m}^* = \sum_{\genfrac{}{}{0pt}{1}{(m,n)\in \Z^2}{(m,n)\not = (0,0)}}$.

\item[-] \textit{Poles and zeros}

The zeros of $z\mapsto {\wpa}(z,\tau)$ are not simply locatable, this is essentially due to the dissymmetric processing of the $z=0$ pole.

The second order poles of ${\wpa}$ are exactly on the lattice $\Lambda_{\tau}$.

\item[-] \textit{Fourier's representations}

For $z = \alpha + \beta\tau$, $(\alpha,\beta)\in\R^2$,

\begin{align}
\ds {\wpa}(z,\tau) &= \ds c_\tau - 4i \sum_{n=1}^{+\infty} \frac{n}{\sin(n\pi\tau)}\cos(n\pi (2z-\tau))\hfill & \ \text{when} \ \beta\in ]0,1[ \label{WpF1} \\
  &= \ds d_\tau + \frac{4i}{\tau^2}\sum_{n=1}^{+\infty} \frac{n}{\sin(n\pi/\tau)}\cos(n(2z-1)\pi/\tau)\hfill & \ \text{when} \ \alpha\in ]0,1[.\label{WpF2}
\end{align}

The terms $c_\tau$ and $d_\tau$ depend only on $\tau$ and check $c_\tau - d_\tau = -\frac{2}{\pi^2 \tau}$. 
Both representations enable one to reconstitute ${\wpa}(z,\tau)$ for $(z,\tau)\in \pa{\C-\Lambda_\tau}\times {\cal H}$ by periodicity.
\ms

Moreover, Euler's equality
\begin{equation}
\forall z\in \C - \Z, \ \ \frac{1}{\sin(\pi z)^2} = \frac{1}{\pi^2} \sum_{n\in\Z} \frac{1}{(z-n)^2} \label{Eul}
\end{equation}
coupled to $(\ref{Wp})$ gives
\begin{equation}
\forall z\in \C - \Lambda_\tau, \ \ {\wpa}(z,\tau) = -\frac{1}{3} + \frac{1}{\sin(\pi z)^2} + \sum_{n=1}^{+\infty} \frac{1}{\sin(\pi(n\tau + z))^2} + \frac{1}{\sin(\pi(n\tau - z))^2} - \frac{2}{\sin(\pi n\tau)^2}.  \label{WpF3}
\end{equation}

\item[-] \textit{Factorization}

Ignorance of the location of the zeros of ${\wpa}$ does not provide a pleasant factorized representation of this function.
\ee
\bs

\textsc{$\bullet$ The $\twpa$ function}
\be
\item[-] \textit{Definition on series form}

For $(z,\tau)\in \pa{\C-(\frac{1+\tau}{2}+\Lambda_\tau)}\times {\cal H}$:
\begin{equation}
{\twpa}(z,\tau) = \frac{1}{\pi^2} \sum_{(n,m)\in\Z^2} \frac{1}{(z+(n+\frac{1}{2})+(m+\frac{1}{2})\tau)^2} - \frac{1}{((n+\frac{1}{2})+(m+\frac{1}{2})\tau)^2}. \label{Wpt}
\end{equation}

\item[-] \textit{Poles and zeros}

The second order zeros of ${\twpa}$ are located on the lattice $\Lambda_{\tau}$.

The poles of ${\twpa}$, also of order two, are on the translated lattice $ \frac{1+\tau}{2} + \Lambda_{\tau}$.

\item[-] \textit{Fourier's representations}

For $z = \alpha + \beta\tau$, $(\alpha,\beta)\in\R^2$,

\begin{align}
\ds {\twpa}(z,\tau) &= \ds 4i \pa{\sum_{n=1}^{+\infty} \frac{(-1)^n n}{\sin(n\pi\tau)} - \sum_{n=1}^{+\infty} \frac{(-1)^n n}{\sin(n\pi\tau)}\cos(2n\pi z)}\hfill & \ \text{when} \ \abs{\beta} < \frac{1}{2} \label{WptF1}\\
  &= \ds \frac{4i}{\tau^2}\pa{-\sum_{n=1}^{+\infty} \frac{(-1)^n n}{\sin(n\pi/\tau)} + \sum_{n=1}^{+\infty} \frac{(-1)^n n}{\sin(n\pi/\tau)}\cos(2n\pi z/\tau)}\hfill & \ \text{when} \ \abs{\alpha} < \frac{1}{2}. \label{WptF2}
\end{align}

Both representations enable one to reconstitute ${\twpa}(z,\tau)$ for $(z,\tau)\in \pa{\C-(\frac{1+\tau}{2}+\Lambda_\tau)}\times {\cal H}$ by periodicity.

In addition, Euler's equality $(\ref{Eul})$ coupled with $(\ref{Wpt})$ gives
\begin{equation}
\forall z\in \pa{\C-(\frac{1+\tau}{2}+\Lambda_\tau)}\times {\cal H}, \ \ {\twpa}(z,\tau) = \sum_{n\in \Z} \frac{1}{\sin(\pi((n+\frac{1}{2})\tau + z + \frac{1}{2}))^2} - \frac{1}{\sin(\pi((n+\frac{1}{2})\tau + \frac{1}{2}))^2}.  \label{WptF3}
\end{equation}

\item[-] \textit{Factorization}

For $(z,\tau)\in \pa{\C-(\frac{1+\tau}{2}+\Lambda_\tau)}\times {\cal H}$ and $q = e^{2i\pi\tau}$:
\begin{equation}
{\twpa}(z,\tau) = -16\sin \left( \pi \,z \right) ^{2} q^{1/2}
\prod_{n=0}^{+\infty} (1-q^{2n+2})^4 (1-q^{n+1/2})^4
\pa{\ds \prod _{n=0}^{+\infty} \frac{\ds \left( 1-{\text{e}^{2\,
i\pi \,z}}{q}^{n+1} \right)  \left( 1-{\text{e}^{-2\,i\pi \,z}}{q}^{n+1}
 \right) }{\left( 1+{\text{e}^{2\,i\pi \,z}}{q}^{n+1/2} \right)  \left( 1+{\text{e}^{-2\,i\pi \,z}}{q}^{n+1/2} \right)}}^2. \label{Facttwpa}
\end{equation}

\ee
\bs

\textsc{$\bullet$ Relationship between functions ${\twpa}$ and ${\wpa}$}
\ms

We will see later, the two functions ${\twpa}$ and ${\wpa}$ lead to modular functions, this aspect is visible on the relations $(\ref{WptF1})$ and $(\ref{WptF2})$, for example.
\ms

For $tau\in {\cal H}$ fixed, knowledge on the poles of ${\wpa}$ and ${\twpa}$ show that $z\mapsto {\twpa}(z,\tau) - {\wpa}(z+\frac{1+\tau}{2},\tau)$ is constant. Liouville's Theorem combined with ${\twpa}(0,\tau) = 0$ gives the following result.

\begin{Lem}
\begin{equation}
{\twpa}(z,\tau) = {\wpa}(z+ \frac{1+\tau}{2},\tau) - {\wpa}(\frac{1+\tau}{2},\tau). \label{Wp2Wpt}
\end{equation}
\end{Lem}
\ms

To complete this lemma, we try to write ${\wpa}$ according to ${\twpa}$. Specifically, we are going to establish the following equality (which is not found in \cite{FeauE})
\begin{Lem}
For $\tau\in {\cal H}$, we have
\begin{equation}
{\wpa}(\frac{1}{2}+\tau,2\tau) = -\frac{1}{3}\pa{{\twpa}(\frac{1}{2},2\tau) + {\twpa}(\tau,2\tau)} \in M_2(\Gamma_0(4)) \label{Mod1}
\end{equation}

and consequently
\begin{equation}
{\wpa}(z,\tau) = {\twpa}\pa{z+\frac{1+\tau}{2},\tau} -\frac{1}{3}\pa{{\twpa}(\frac{1}{2},\tau) + {\twpa}(\frac{\tau}{2},\tau)}. \label{Wpt2Wp}
\end{equation}
\end{Lem}

Note that this last relationship leads to a partial factorization of ${\wpa}(z,\tau)$. It also enables one to specify the constants $c_\tau$ and $d_\tau$ in the $(\ref{WpF1})$ and $(\ref{WpF2})$ equalities.
\ms

\begin{proof}
The relation $(\ref{Wpt2Wp})$ is clearly deduced from the equalities $(\ref{Wp2Wpt})$ and $(\ref{Mod1})$, then the modular character of $(\ref{Mod1})$ will be established in the following paragraph.
\ms

It remains to prove the equality $(\ref{Mod1})$, we rely for that on the representations $(\ref{WpF3})$ and $(\ref{WptF3})$. We find
\[
\begin{array}{lcl}
\ds {\wpa}(\frac{1}{2}+\tau,2\tau) & = &\ds -\frac{1}{3} + \frac{1}{\cos(\pi\tau)^2} + \sum_{n=1}^{+\infty} \frac{1}{\cos((2n+1)\pi\tau)^2} + \frac{1}{\cos((2n-1)\pi\tau)^2} - \frac{2}{\sin(2n\pi\tau)^2}\\
 & = & \ds -\frac{1}{3} + 2 \sum_{n=0}^{+\infty} \frac{1}{\cos((2n+1)\pi\tau)^2} - 2 \sum_{n=1}^{+\infty} \frac{1}{\sin(2n\pi\tau)^2}.
\end{array}
\]
On the other hand,
\[
\begin{array}{lcl}
\ds {\twpa}(\frac{1}{2},2\tau) + {\twpa}(\tau,2\tau) & = &\ds \sum_{n\in\Z} \frac{1}{\sin((2n+1)\pi\tau)^2} - \sum_{n\in\Z} \frac{1}{\cos((2n+1)\pi\tau)^2} \\
& & \ds  + \sum_{n\in\Z} \frac{1}{\cos((2n+2)\pi\tau)^2} - \sum_{n\in\Z} \frac{1}{\cos((2n+1)\pi\tau)^2}\\
& = & \ds 1 + 2\sum_{n=0}^{+\infty} \frac{1}{\sin((2n+1)\pi\tau)^2} + \sum_{n=1}^{+\infty} \frac{1}{\cos((2n)\pi\tau)^2}\\
& & \ds \ - 4 \sum_{n=0}^{+\infty} \frac{1}{\cos((2n+1)\pi\tau)^2}\\
& = & \ds 1 - 2\sum_{n=1}^{+\infty} \frac{1}{\sin(2n\pi\tau)^2} + 2\sum_{n=1}^{+\infty} \frac{1}{\sin(n\pi\tau)^2}\\
& & \ds \ + 2\sum_{n=1}^{+\infty} \frac{1}{\cos((2n)\pi\tau)^2} - 6 \sum_{n=0}^{+\infty} \frac{1}{\cos((2n+1)\pi\tau)^2}\\
& = & \ds 1 + 2\sum_{n=1}^{+\infty} \frac{1}{(\sin(n\pi\tau)\cos(n\pi\tau))^2}  - 2\sum_{n=1}^{+\infty} \frac{1}{\sin(2n\pi\tau)^2}\\
& & \ds \ - 6 \sum_{n=0}^{+\infty} \frac{1}{\cos((2n+1)\pi\tau)^2}\\
& = & \ds 1 +6\sum_{n=1}^{+\infty} \frac{1}{\sin(2n\pi\tau)^2} - 6 \sum_{n=0}^{+\infty} \frac{1}{\cos((2n+1)\pi\tau)^2}\\
& = & \ds -3 \, {\wpa}(\frac{1}{2}+\tau,2\tau)
\end{array}
\]
and the result.
\end{proof}
\bs
\ms

\section{-- The symmetries of ${\twpa}$ and ${\wpa}$; Examples of modular forms}

The equalities $(\ref{Wp})$, $(\ref{WptF1})$ and $(\ref{WptF2})$ for ${\twpa}$ and $(\ref{Wp})$ for ${\wpa}$ lead to the following elementary symmetries:

\begin{equation}
\begin{array}{c|c}
{\twpa}(z+1,\tau) = {\twpa}(z+\tau,\tau) = {\twpa}(z,\tau) \ & \ \  {\wpa}(z+1,\tau) = {\wpa}(z+\tau,\tau) = {\wpa}(z,\tau)\\
{\twpa}(z,\tau+2) = {\twpa}(z,\tau) & {\wpa}(z,\tau+1) = {\wpa}(z,\tau)\\
{\twpa}(-z,\tau) = {\twpa}(z,\tau) & {\wpa}(-z,\tau) = {\wpa}(z,\tau) \\
{\twpa}(z,-\tau) = {\twpa}(z,\tau) & {\wpa}(z,-\tau) = {\wpa}(z,\tau) 
\end{array}
\end{equation}

Let us note two points. For $z\in \C$ fixed, the period of ${\wpa}(z,.)$ is worth $1$, whereas it is worth $2$ for ${\twpa}(z,.)$. Besides, the last symmetry is correct, but from the point of view of modular forms, we only consider $\tau \in {\cal H}$ case. To this is added the inversion symmetry which leads to the modular character:
\ms

%\textit{L'inversion:}
\begin{equation}
\begin{array}{c|c}
\ds {\twpa}(z,\tau) = \frac{1}{\tau^2}{\twpa}\pa{\frac{z}{\tau},-\frac{1}{\tau}} \ \ & \ \ \ds {\wpa}(z,\tau) = \frac{1}{\tau^2}{\wpa}\pa{\frac{z}{\tau},-\frac{1}{\tau}}
\end{array}
\end{equation}

Now let us show the modularity of some modular forms built with $\wpa$ and $\twpa$.
\ms

For a given $N$, we know systems that generate $(\Gamma_0(N),\times)$. A standard result indicates that it is sufficient to verify the $(i)$ condition of definition I-1.3 on such a system to deduct the $(i)$ property in its generality. Similarly, we know how to build a representative system for each of the classes of $\Gamma_0(1)/\Gamma_0(N)$. To show $(ii)$, just check this condition on the equivalency class representatives.
\bs

We will give, in this paragraph, examples of modular forms to build explicit bases. In order not to overextend this article, it will not be possible to produce a demonstration of modularity for each of these functions. Nevertheless, we will now show on some examples, how the symmetries of $\wpa$ and $\twpa$ lead to demonstrations of modularity. An effective way to verify future and unproven modularity claims is to compare with the $q = e^{2i\pi\tau}$ asymptotic expansion of modular feature bases found, particularly in SAGE software.
\bs
\ms

\textsc{$\bullet$ $\ds {\wpa}\pa{\tau,2\tau} \in M_{2}(\Gamma_0(2))$}
\ms

This result is well known, see \cite{Diam} p. 130 for example
\ms

From the relationships $(\ref{Wpt})$, $(\ref{WptF1})$ and $(\ref{WptF2})$, the conditions of definition I-1.3 are easy to verify except the condition $(i)$ for the non-trivial generator of $\Gamma_0(2)$, $\begin{pmatrix} 1 & 0\\ 2 & 1\end{pmatrix}$.

The cusps behavior of a function $\Phi$ is processed using a system of class representatives of $\Gamma_0(1)/\Gamma_0(2)$ given by SAGE:
\[\begin{pmatrix}1& 0\\ 0& 1\end{pmatrix}, \
\begin{pmatrix}0& -1\\ 1& 0\end{pmatrix} \ \text{et} \ 
\begin{pmatrix}1& 0\\ 1& 1\end{pmatrix}.\]

For such a matrix it is necessary to show that $\tau \mapsto (c\tau+d)^{-2}\Phi\pa{\frac{a\tau+b}{c\tau+d}}$ admits a limit when $\tau$ tends towards $i\infty$ and this is clear on the representation of Weierstrass given in $(\ref{Wp})$.
\bs

Let us pose $\Phi(\tau) = \wpa(\tau,2\tau)$ and check the modularity relationship for $\gamma$ by applying the inversion twice:

\begin{equation}
\begin{array}{r c l}
\ds \Phi\pa{\frac{\tau}{2\tau+1}} & = & \ds {\wpa}\pa{\frac{\tau}{2\tau+1},\frac{2\tau}{2\tau+1}} \\
 & = & \ds \pa{\frac{2\tau+1}{2\tau}}^{2}{\wpa}\pa{\frac{1}{2},-\frac{1}{2\tau}} \\
 & = & \ds \pa{\frac{2\tau+1}{2\tau}}^{2}(-2\tau)^2{\wpa}\pa{-\tau,2\tau} \\
 & = & \ds (2\tau+1)^{2}\Phi(\tau).
\end{array} \label{M22}
\end{equation}
\ms

\textsc{$\bullet$ $\ds {\twpa}\pa{\frac{1}{2},2\tau} \in M_{2}(\Gamma_0(4))$}
\ms

Relationships $(\ref{Wpt})$ and $(\ref{WptF1})$ give immediately for $\in {\cal H}$:
\begin{equation}
\begin{array}{lcl}
\Phi(\tau) = {\twpa}(\frac{1}{2},2\tau) & = & \ds \frac{1}{\pi^2} \sum_{(n,m)\in\Z^2} \frac{1}{(n+(2m+1)\tau)^2} - \frac{1}{((n+\frac{1}{2})+(2m+1)\tau)^2} \label{M11}\\
& = & \ds -8i \sum_{n=1}^{+\infty} \frac{2n+1}{\sin(2(2n+1)\pi\tau)}.
\end{array}
\end{equation}

Similarly, the conditions in definition I-1.3 are easy to establish except the $(i)$ condition for the non-trivial generator $\gamma = \begin{pmatrix} 1 & 0\\ 4 & 1\end{pmatrix}$, and possibly the cusps behavior.
\ms

Let us check the modularity relationship for $\gamma$ by applying the inversion twice:

\begin{equation}
\begin{array}{r c l}
\ds \Phi\pa{\frac{\tau}{4\tau+1}} & = & \ds {\twpa}\pa{\frac{1}{2},\frac{2\tau}{4\tau+1}} \\
 & = & \ds \pa{\frac{4\tau+1}{2\tau}}^{2}{\twpa}\pa{\frac{1}{2} \frac{4\tau+1}{2\tau},-\frac{4\tau+1}{2\tau}} \\
 & = & \ds \pa{\frac{4\tau+1}{2\tau}}^{2}{\twpa}\pa{\frac{1}{4\tau} + 1,-\frac{1}{2\tau} - 2} \\
 & = & \ds \pa{\frac{4\tau+1}{2\tau}}^{2}{\twpa}\pa{\frac{1}{4\tau},-\frac{1}{2\tau}} \\
 & = & \ds \pa{\frac{4\tau+1}{2\tau}}^{2} \pa{2\tau}^{2}{\twpa}\pa{-\frac{1}{2},2\tau} \\
 & = & \ds (4\tau+1)^{2}\Phi(\tau).
\end{array} \label{M24}
\end{equation}
\ms

For the study at cusps of a function $\Phi$, we use the system of class representatives of $\Gamma_0(1)/\Gamma_0(4)$ given by SAGE:
\[\begin{pmatrix}1& 0\\ 0& 1\end{pmatrix}, \
\begin{pmatrix}0& -1\\ 1& 0\end{pmatrix}, \ 
\begin{pmatrix}1& 0\\ 1& 1\end{pmatrix}, \
\begin{pmatrix}0& -1\\ 1& 2\end{pmatrix}, \
\begin{pmatrix}0& -1\\ 1& 3\end{pmatrix}, \
\begin{pmatrix}1& 0\\ 2& 1\end{pmatrix}.\]

For these matrices, the existence of a limit for $\tau \mapsto (c\tau+d)^{-2}\Phi\pa{\frac{a\tau+b}{c\tau+d}}$ when $\tau$ tends towards $i\infty$ is clear on the Weierstrass representation given in $(\ref{M11})$.
\bs

\textsc{$\bullet$ $\ds {\twpa}\pa{\tau,2\tau} \in M_{2}(\Gamma_0(4))$}
\ms

Using $(\ref{Wpt})$ but also $(\ref{WptF2})$ we get:

\begin{equation}
\begin{array}{lcl}
\ds \Phi(\tau) = {\twpa}(\tau,2\tau) & = & \ds \frac{1}{\pi^2} \sum_{(n,m)\in\Z^2} \frac{1}{((n+\frac{1}{2})+2m\tau)^2} - \frac{1}{((n+\frac{1}{2})+(2m+1)\tau)^2} \\
 & = & \ds \frac{2i}{\tau^2} \sum_{n=1}^{+\infty} \frac{n} {\sin(n\pi/\tau)}. \label{M244}
 \end{array}
\end{equation}
\ms

Again, the conditions of definition I-1.3 are easy to establish except the $(i)$ condition for the non-trivial generator of $\Gamma_0(4)$, $A = \begin{pmatrix} 1 & 0\\ 4 & 1\end{pmatrix}$. Cusps behavior is treated as above thanks to the Weierstrass form of $(\ref{M244})$ relationships.
\ms

\begin{equation}
\begin{array}{r c l}
\ds \Phi\pa{\frac{\tau}{4\tau+1}} & = & \ds {\twpa}\pa{\frac{\tau}{4\tau+1},\frac{2\tau}{4\tau+1}} \\
 & = & \ds \pa{\frac{4\tau+1}{2\tau}}^{2}{\twpa}
 \pa{\frac{1}{2},-\frac{4\tau+1}{2\tau}} \\
 & = & \ds \pa{\frac{4\tau+1}{2\tau}}^{2} {\twpa}\pa{\frac{1}{2},-\frac{1}{2\tau}} \\
 & = & \ds \pa{\frac{4\tau+1}{2\tau}}^{2} \pa{-2\tau}^{2}{\twpa}\pa{-\tau,2\tau} \\
 & = & \ds (4\tau+1)^{2}\Phi(\tau).
\end{array}
\end{equation}
\bs

\textsc{$\bullet$ $\ds {\twpa}\pa{\tau,2\tau}^2 \in M_{4}(\Gamma_0(2))$}
\ms

Let us pose $\Phi(\tau) = \ds\pa{\tau,2\tau}$. Based on the previous case, $\Psi = \Phi^2\in M_{4}(\Gamma_0(4))$ and
we want to show better: $\Psi \in M_{4}(\Gamma_0(2))$.

Again, the conditions of definition I-1.3 are easy to establish except the $(i)$ condition for the non-trivial generator of $\Gamma_0(2)$, $A = \begin{pmatrix} 1 & 0\ 2 & 1\end{pmatrix}$. Cusps behavior is treated as above thanks to the $\Psi$ Weierstrass representation.
\ms

In order to check the modularity $\Psi$ under the action of $A$, let us notice the equality $\twpa(\frac{1}{2},\tau-1) = -\twpa(\frac{1}{2},\tau))$ which is an immediate consequence of $(\ref{WptF1})$.
\ms

\begin{equation}
\begin{array}{r c l}
\ds \Psi\pa{\frac{\tau}{2\tau+1}} & = & \ds {\twpa}\pa{\frac{\tau}{2\tau+1},\frac{2\tau}{2\tau+1}}^2 \\
 & = & \ds \pa{\frac{2\tau+1}{2\tau}}^{4}{\twpa}
 \pa{\frac{1}{2},-\frac{2\tau+1}{2\tau}}^2 \\
 & = & \ds \pa{\frac{2\tau+1}{2\tau}}^{4} \pa{-{\twpa}\pa{\frac{1}{2},-\frac{1}{2\tau}}}^2 \\
 & = & \ds \pa{\frac{2\tau+1}{2\tau}}^{4} \pa{-2\tau}^{2}{\twpa}\pa{-\tau,2\tau} \\
 & = & \ds (2\tau+1)^{4}\Psi(\tau).
\end{array}
\end{equation}

The result can be deduced from this.\bs

\textsc{$\bullet$ $\ds {\wpa}(\frac{1}{2}+\tau,2\tau) = -\frac{1}{3}\pa{{\twpa}(\frac{1}{2},2\tau) + {\twpa}(\tau,2\tau)} \in M_2(\Gamma_0(4))$}
\ms

This is the $(\ref{Mod1})$ relationship already established and which is also a modular equality, here is why.
\ms

First, we check that $M_2(\Gamma_0(4))$ belongs to the space by checking the only non-trivial condition: $\Phi\pa{\frac{\tau}{4\tau+1}} = (4\tau+1)^{2}\Phi(\tau)$.
\ms

\begin{equation}
\begin{array}{r c l}
\ds \Phi\pa{\frac{\tau}{4\tau+1}} & = & \ds {\wpa}\pa{\frac{6\tau+1}{8\tau+2},\frac{2\tau}{4\tau+1}} \\
 & = & \ds \pa{\frac{4\tau+1}{2\tau}}^{2}{\wpa}
 \pa{\frac{6\tau+1}{4\tau},-\frac{4\tau+1}{2\tau}} \\
 & = & \ds \pa{\frac{4\tau+1}{2\tau}}^{2} {\wpa}\pa{\frac{2\tau+1}{4\tau},-\frac{1}{2\tau}} \\
 & = & \ds \pa{\frac{4\tau+1}{2\tau}}^{2} \pa{-2\tau}^{2}{\wpa}\pa{\frac{2\tau+1}{2},2\tau} \\
 & = & \ds (4\tau+1)^{2}\Phi(\tau).
\end{array}
\end{equation}
\ms

Besides, $\ds {\twpa}(\tau,2\tau)$ and ${\twpa}(\frac{1}{2},2\tau))$ are two elements of the two dimensional space $M_2(\Gamma_0(4))$.

The $(\ref{M11})$ and $(\ref{M22})$ relationships enable one to write:
\[{\twpa}(\tau,2\tau) = 1 - 8q + 24q^2 + O(q^3) \ \ \text{et} \ \ {\twpa}(\frac{1}{2},2\tau) = -16 q + O(q^3).\]
\ms

So $ ({\twpa}(\tau,2\tau), -\frac{1}{16}{\twpa}(\frac{1}{2},2\tau))$ is a $M_4(\Gamma_0(2))$ unitary upper triangular basis.

As a result $\ds {\wpa}(\frac{1}{2}+\tau,2\tau) = -\frac{1}{3}(1-24q+24q^2+O(q^3))$ is an element of $M_4(\Gamma_0(2))$ and the relationship $(\ref{Mod1})$ is the expression of $\ds {\wpa}(\frac{1}{2}+\tau,2\tau)$ in the previous basis.
\bs

\textsc{$\bullet$ $\ds {\twpa}\pa{\frac{1}{2},\tau}^2 \in M_{4}(\Gamma_0(2))$}
\ms

This function is particularly interesting since it will be, with a multiplicative constant, $\Delta_2$.
\ms

The relationships $(\ref{Wpt})$ and $(\ref{WptF1})$ give:
\begin{equation}
\begin{array}{lcl}
\Phi(\tau) = {\twpa}(\frac{1}{2},\tau)^2 & = & \ds \frac{1}{\pi^2} \pa{\sum_{(n,m)\in\Z^2} \frac{1}{(n+(m+\frac{1}{2})\tau)^2} - \frac{1}{((n+\frac{1}{2})+(m+\frac{1}{2})\tau)^2}}^2\\
& = & \ds -64 \pa{\sum_{n=1}^{+\infty} \frac{2n+1}{\sin((2n+1)\pi\tau)}}^2. \label{M32}
\end{array}
\end{equation}

The $\Phi$ function is clearly $1$-periodic, it remains to check the modular relationship $\ds \Phi\pa{\frac{\tau}{4\tau+1}} = (4\tau+1)^{2}\Phi(\tau)$:
\ms

\begin{equation}
\begin{array}{r c l}
\ds \Phi\pa{\frac{\tau}{4\tau+1}} & = & \ds {\twpa}\pa{\frac{1}{2},\frac{\tau}{4\tau+1}}^2 \\
 & = & \ds \pa{\frac{4\tau+1}{\tau}}^{4}{\twpa}\pa{\frac{1}{2} \frac{4\tau+1}{\tau},-\frac{4\tau+1}{\tau}}^2 \\
 & = & \ds \pa{\frac{4\tau+1}{\tau}}^{4}{\twpa}\pa{\frac{1}{2\tau} + 2,-\frac{1}{\tau} - 4}^2 \\
 & = & \ds \pa{\frac{4\tau+1}{\tau}}^{4}{\twpa}\pa{\frac{1}{2\tau},-\frac{1}{\tau}}^2 \\
 & = & \ds \pa{\frac{4\tau+1}{\tau}}^{4} \tau^{2}{\twpa}\pa{-\frac{1}{2},\tau}^2 \\
 & = & \ds (4\tau+1)^{4}\Phi(\tau).
\end{array} \label{ModD42}
\end{equation}
\ms

This check suggests that the root of $\Delta_2$ could be a modular form, but in fact not: the function $\ds {\twpa}\pa{\frac{1}{2},\tau}$ is 2-periodic but not 1-periodic as we can easily see on its representations.
\bs
\ms

\section{-- Determination of an element $E_{2,N}^{(0)}$}
\ms

Part I Theorem I-5.2 noted the existence of a $0$ valuation element in $M_{2}(\Gamma_0(p))$ and its importance in the algorithmic description of spaces $(M_{2k}(\Gamma_0(p)))_{k\in\N^*}$ for $N \se 2$.

The proposed demonstration was based on Eisenstein's $E_{2}$ series. We will give a different demonstration here and provide an explicit form for an $E_{2,N}^{(0)}$ element.

\begin{Thm}
Either $N\in\N^*$, then the function defined by
\begin{equation}\label{Polsym1}
\Phi_N(\tau) = \frac{-3}{N-1} \sum_{k=1}^{N-1}{\wpa}(k\tau,N\tau)
\end{equation}
is an element of $M_{2}(\Gamma_0(N))$ unitary with valuation $0$.
\end{Thm}

\begin{proof}
According to the $(\ref{WpF3})$ relationship, $\Phi_N$ is $1$-periodic and
\begin{equation}
\forall k\in \llbracket1,N-1\rrbracket, \ \ \lim_{\tau\to i\infty} {\wpa}(k\tau,N\tau) = - \frac{1}{3}.
\end{equation}
The result is $\Phi_N = 1 + O(q)$ and $\nu(\Phi_N) = 0$.\ms

The existence of a limit in any cusp results from $(\ref{Wp})$, it remains to study the action of $\Gamma_0(N)$ on $\Phi_N$.

\begin{equation}
\begin{array}{lcl}
\ds(N-1)\pi^2 \Phi_N(\tau) & = & \ds \pa{\sum_{k=1}^{N-1}\frac{1}{k^2}}\frac{1}{\tau^2} + \sum_{n,m}^* \sum_{k=1}^{N-1} \pa{\frac{1}{(n+(k+mN)\tau)^2} - \frac{1}{(n+mN\tau)^2}}
\end{array}
\end{equation}

Let $\begin{pmatrix} a & b\\ c & d\end{pmatrix}$ be in $\Gamma_0(N)$, we pose $c = Nc'$, $\ds \Psi_N(\tau) = (N-1)\pi^2(c\tau+d)^{-2} \Phi_N\pa{\frac{a\tau+b}{c\tau+d}}$ and:

\begin{equation}
\begin{array}{lcl}
\ds\Psi_N(\tau) & = & \ds \pa{\sum_{k=1}^{N-1}\frac{1}{k^2}}\frac{1}{(a\tau+b)^2} + \sum_{n,m}^* \sum_{k=1}^{N-1} \Big(\frac{1}{((c\tau+d)n+(k+mN)(a\tau+b))^2} - \frac{1}{((c\tau+d)n+mN(a\tau+b))^2}\Big)\\
& = &  \ds \pa{\sum_{k=1}^{N-1}\frac{1}{k^2}}\frac{1}{(a\tau+b)^2} + \sum_{n,m}^* \sum_{k=1}^{N-1} \Big(\frac{1}{((ak + amN+cn)\tau + (bk+bmN + dn))^2}\\
& & \ds {\hskip 7cm} - \frac{1}{(( amN+cn)\tau + (bmN + dn))^2}\Big)\\
\end{array}
\end{equation}
\ms

The transformation
$\ds\left\{
\begin{array}{lcl}
p & = & dn+Nbm\\
q & = & c'n+am
\end{array}\right.$
is an isomorphism of the lattice $Z^2$ with
$\ds\left\{
\begin{array}{lcl}
n & = & ap-Nbq\\
m & = & -c'p+dq
\end{array}\right.$.
As a result,

\begin{equation}
\begin{array}{lcl}
\ds \Psi_N(\tau) & = & \ds \pa{\sum_{k=1}^{N-1}\frac{1}{k^2}}\frac{1}{(a\tau+b)^2} + \sum_{p,q}^* \sum_{k=1}^{N-1} \pa{\frac{1}{((ak + qN)\tau + (bk+p))^2} - \frac{1}{(qN\tau + p)^2}}.
\end{array}
\end{equation}

We then notice that

\begin{equation}
\begin{array}{lcl}
\ds S_{N}(\tau) & = & \ds \pa{\sum_{k=1}^{N-1}\frac{1}{k^2}}\pa{\frac{1}{(a\tau+b)^2} - \frac{1}{(a\tau)^2}}
 + \sum^* \sum_{k=1}^{N-1} \frac{1}{((ak + qN)\tau + (bk+p))^2} - \frac{1}{((ak + qN)\tau + p)^2}\\
& = & \ds \sum_{(p,q)\in \Z^2} \sum_{k=1}^{N-1} \frac{1}{((ak + qN)\tau + (bk+p))^2} -  \frac{1}{((ak + qN)\tau + p)^2}\\
& = & 0
\end{array}
\end{equation}
with a licit variable change on $p$, and therefore

\begin{equation}
\begin{array}{lcl}
\ds \Psi_N(\tau) & = & \ds \ds \pa{\sum_{k=1}^{N-1}\frac{1}{k^2}}\pa{\frac{1}{(a\tau)^2}} + \sum_{p,q}^* \sum_{k=1}^{N-1} \frac{1}{((ak + qN)\tau + p)^2} - \frac{1}{(qN\tau + p)^2}
\end{array}
\end{equation}

For each $k\in \llbracket1,N-1\rrbracket$ $ak = u_k q + r_k$ with $(u_k,r_k)\in \Z\times \llbracket0,N-1\rrbracket$. We also know that $ad-c'Nb = $1, and therefore $r_k \not= $0 and even $\{r_k, \ 1\ie k\ie N-1\} = \{1,\ldots,N-1\}$ because $\overline{a}$ is invertible in $Z/NZ$. You can also write: for $ r\in \llbracket1,N-1\rrbracket$ there is a single $(v_r,k_r)\in \Z\times \llbracket1,N-1\rrbracket$ such that $r = v_r N + ak_r$. By reordering,

\begin{equation}
\begin{array}{lcl}
\ds \Psi_N(\tau) & = & \ds \ds \pa{\sum_{k=1}^{N-1}\frac{1}{k^2}}\pa{\frac{1}{(a\tau)^2}} + \sum_{p,q}^* \sum_{r=1}^{N-1} \frac{1}{((r + (q-v_r)N)\tau + p)^2} - \frac{1}{(qN\tau + p)^2}\\
 & = & \ds \pa{\sum_{k=1}^{N-1}\frac{1}{k^2}}\pa{\frac{1}{(a\tau)^2}} + \sum_{r=1}^{N-1} \sum_{p,q}^* \frac{1}{((r + (q-v_r)N)\tau + p)^2} - \frac{1}{(qN\tau + p)^2}
\end{array}
\end{equation}

If $p\not=0$, $\ds \sum_{q} \frac{1}{((r + (q-v_r)N)\tau + p)^2} - \frac{1}{(qN\tau + p)^2} = \sum_{q} \frac{1}{((r + qN)\tau + p)^2} - \frac{1}{(qN\tau + p)^2}$

If $p=0$,
\begin{equation*}
\begin{array}{lcl}
\ds \sum_{q\not=0} \frac{1}{((r + (q-v_r)N)\tau)^2} - \frac{1}{(qN\tau)^2} & = & \ds \frac{1}{(r\tau)^2}-\frac{1}{((r - v_rN)\tau)^2}+\sum_{q\not=0} \frac{1}{((r + qN)\tau)^2} - \frac{1}{(qN\tau)^2}\\
& = & \ds \frac{1}{(r\tau)^2}-\frac{1}{(ak_r\tau)^2}+\sum_{q\not=0} \frac{1}{((r + qN)\tau)^2} - \frac{1}{(qN\tau)^2}
\end{array}
\end{equation*}

So, since $\{k_1,\ldots,k_{N-1}\} = \{1,\ldots,N-1\}$,

\begin{equation}
\begin{array}{lcl}
\ds \Psi_N(\tau) & = & \ds \pa{\sum_{k=1}^{N-1}\frac{1}{k^2}}\frac{1}{\tau^2} +  \sum_{p,q}^* \sum_{r=1}^{N-1} \frac{1}{((r + qN)\tau + p)^2} - \frac{1}{(qN\tau + p)^2}\\
& = & \ds (N-1)\Phi_N(\tau)
\end{array}
\end{equation}

which is the result of modularity sought.
\end{proof}
\bs

The interested reader may seek to verify the following equality:
\begin{equation}
G_{2,N}(\tau) = \pi^2 \sum_{k=1}^{N-1}{\wpa}(k\tau,N\tau) = \sum_{k=1}^{N-1}{\wp}(k\tau,N\tau)
\end{equation}
where $G_{2,N}(\tau)$ is the exceptional Eisenstein series reminded in Part I, (see \cite{Miya} or \cite{Diam}). Note that for $N = 2$, $3$, $5$ and $7$, it is a consequence of the relationship $\dim(M_2(\Gamma_0(N))) = 1$.
\bs

\begin{Thm}
Either $N\se 1$, we have the representation of Fourier
\begin{equation}
\frac{-3}{N-1} \sum_{k=1}^{N-1}{\wpa}(k\tau,N\tau) = 1 - \frac{6}{N-1}\sum_{n=1}^{+\infty} \frac{1}{\sin(n\pi\tau)^2} + \frac{6N}{N-1}\sum_{n=1}^{+\infty} \frac{1}{\sin(n\pi N\tau)^2}.
\end{equation}
\end{Thm}

Note that we can provide a more numerically efficient version of $\Phi_N$ by remembering that $z\mapsto {\wpa}(z,N\tau)$ is an elliptic function of periods $1$ and $N\tau$. Therefore, for $k\in \llbracket1,N-1\rrbracket$, ${\wpa}(k\tau,N\tau) = {\wpa}((N-k)\tau,N\tau)$. Thus, according to the parity of $N$:
\begin{equation}
\Phi_N(\tau) = 
\left\{
\begin{array}{ll}
\ds \frac{1}{n} \sum_{k=1}^n {\wpa}(k\tau,(2n+1)\tau) & \text{when} \ N = 2n+1\\
\ds \frac{1}{n-\frac{1}{2}}\pa{\frac{1}{2}{\wpa}(n\tau,2n\tau) + \sum_{k=1}^{n-1} {\wpa}(k\tau,(2n)\tau)} & \text{when} \ N = 2n.
\end{array}
\right.
\end{equation}
\bs

Using the ratings from Part I for generic bases ${\cal B}_{2k}(\Gamma_0(N)) = (E_{2k,N}^{(s)})_{0\ie s \ie d_{2k}(N)-1}$, we can choose:
\begin{Cor}
\begin{equation}
E_{2,2}^{(0)} = -3 {\wpa}(\tau,2\tau) = 1+24q+24q^2+96q^3+24q^4+144q^5+96q^6+O(q^7) \in M_2(\Gamma_0(2)),
\end{equation}
\begin{equation}
E_{2,3}^{(0)} = -3 {\wpa}(\tau,3\tau) = 1+12q+36q^2+12q^3+84q^4+72q^5+36q^6+O(q^7) \in M_2(\Gamma_0(3)).
\end{equation}
\end{Cor}

%\begin{proof}
%Le résultat modulaire provient du théorème I-7.1. Les développement en série de Fourier proviennent du théorème I-7.2, aidé de la remarque précédente dans le cas $N=3$ puisque ${\pi^2\wpa}(\tau,3\tau) = {\pi^2\wpa}(2\tau,3\tau)$.
%\end{proof}
%\bs
\bs

To finish this paragraph, and without any demonstration, we note that $(\ref{Polsym1})$ is an access point to a much more general toolbox.
\ms

If $P\in \C[X_1,\ldots,X_{N-1}]$ is a homogeneous symmetric polynomial of $k$ degree, then the application

\[\tau \mapsto P\pa{{\wpa}(\tau,N\tau), \ldots, {\wpa}((N-1)\tau,N\tau)}\]

is an element of $M_{2k}(\Gamma_0(N))$. This process will be used in the $N = 7$ case.
\ms

Just check that for everything $k\in \N^*$, $\tau \mapsto \ds \sum_{n=1}^{N-1}{\wpa}(n\tau,N\tau)^k$ belongs to $M_{2k}(\Gamma_0(N))$, thanks to Newton's formulae on symmetric polynomials.
\ms
\bs
\ms

\section{ -- Structure and bases of $(M_{2k}(\Gamma_0(N)))_{k\in \N^*}$, $1\ie N \ie 10$}
\ms

In this paragraph, we are trying to determine explicit unitary upper triangular bases for all spaces $(M_{2k}(\Gamma_0(N)))_{k\in \N^*}$, $1\ie N \ie 10$. By explicit, we mean representable using $\eta$-products or elliptic functions $\wpa$ and $\twpa$, which naturally provides other representations allowing serial Fourier development at any order.\ms

Using the notations from part I, $d_{2k}(N)$ is the dimension of $M_{2k}(\Gamma_0(N))$, and by ${\cal B}_{2k}(\Gamma_0(N)) = (E_{2k,N}^{(s)})_{0\ie s \ie d_{2k}(N)-1}$ a $M_{2k}(\Gamma_0(N))$ unitary upper triangular basis. Such a basis exists but is not unique, unlike the strictly increasing sequence of valuations $(\nu(E_{2k,N}^{(s)}))_{0\ie s \ie d_{2k}(N)-1}$.
\ms

Theorem I-7.3 in the first part indicates that bases for spaces $(M_{2k}(\Gamma_0(N)))_{k\se \frac{1}{2}\rho_N+1}$ are calculable as soon as we know bases $({\cal B}_{2k}(\Gamma_0(N)))_{k\ie \frac{1}{2}\rho_N}$ as well as $(E_{\rho_N+2,N}^{(s)})_{0\ie s \ie \nu(\Delta(N))-1}$, the beginning of a ${\cal B}_{\rho_N+2}(\Gamma_0(N))$ basis.
\ms

When $d_{\rho_N}(N) = \nu(\Delta_N) +1$, and it will be the case for $1\ie N \ie 10$, it is enough to know $({\cal B}_{2k}(\Gamma_0(N)))_{k\ie \frac{1}{2}\rho_N}$. We can indeed choose $E_{\rho_N+2,N}^{(s)} = E_{\rho_N,N}^{(s)} E_{2,N}^{(0)}$ for $0\ie s \ie \nu(\Delta_N)-1$.
\ms

Theorem I-7.3 then takes the following form:

\begin{Thm}\label{ThmStruct101}
Let $N$ be a positive integer such as $d_{\rho_N}(N) = \nu(\Delta_N) +1$, then
\begin{equation}
\forall k\se \frac{\rho_N}{2}, \ \ M_{2k}(\Gamma_0(N)) = \Delta_N.M_{2k-\rho_N}(\Gamma_0(N)) \oplus Vect\pa{E_{\rho_N,N}^{(s)}[E_{2,N}^{(0)}]^{k-\frac{\rho_N}{2}} \ / \ 0\ie s < \nu(\Delta_N)}.
\end{equation}
Therefore, if $k\in \N^*$ and $k = q\frac{\rho_N}{2}+r$ with $1\ie r \ie \frac{\rho_N}{2}$,

\begin{equation}
 M_{2k}(\Gamma_0(N)) = \Delta_N^{q}. M_{2r}(\Gamma_0(N)) \bigoplus_{n=0}^{q-1}  \Delta_N^n .Vect\pa{E_{\rho_N,N}^{(s)}[E_{2,N}^{(0)}]^{k-(n+1)\frac{\rho_N}{2}} \ / \ 0\ie s < \nu(\Delta_N)}.\label{Base1}
\end{equation}
\end{Thm}

%L'égalité $(\ref{Base1})$ permet de construire les bases $({\cal B}_{2k}(\Gamma_0(N)))_{k \se \rho_N+1}$ à partir de $({\cal B}_{2k}(\Gamma_0(N)))_{k \ie \rho_N}$.

The characteristic value $N = 11$ is the first for which $d_{\rho_N}(N) < \nu(\rho_N)+1$, and in this case we must of course apply Theorem I-7.3. 
Note that, quickly, the hypothesis $d_{\rho_N}(N) = \nu(\rho_N)+$1 becomes systematically false.
\bs
\bs
\ms

\subsection{ -- Structure and bases of $(M_{2k}(\Gamma_0(1)))_{k\in \N^*}$}
\ms

The $N = 1$ case is well known \cite{Serre,Stein,Ono}, it differs from the generic  $N\se 2$ case because $M_{2}(\Gamma_0(1)) = \{0\}$.
The strong modular unit is the modular discriminant $\Delta_1(\tau) = \Delta(\tau) = \eta(\tau)^{24}$, the first space dimensions table is the following:

\begin{center}
\begin{tabular}{c|c|c|c|c|c|c|c|c}
$2k$   & 2  & 4  & 6  & 8  & 10  & 12  & 14  & 16\\
\hline
$d_{2k}(1)$   & 0  & 1  & 1  & 1  & 1  & 2  & 1  & 2 \\
\end{tabular}
\end{center}
\ms

We have classically
\begin{equation*}
\forall k\se 7, \ \ M_{2k}(\Gamma_0(1)) = Vect(E_{2k,1}^{(0)})\oplus \Delta . M_{2k-12}(\Gamma_0(1)).
\end{equation*}

As for the bases, the reference result built on the Eisenstein series (see \cite{Stein}, for example) is: the space $M_{2k}(\Gamma_0(1))$ is based on the family of $E_4^a E_6^b$, with $(a,b)\in\N^2$ such that $2a+3b = k$. However, this basis is not upper triangular, all these elements are of zero valuation.
\ms

The results of Theorem II-\ref{ThmStruct101} enable one recursively constructing unitary upper triangular bases, even in the absence of an $E_{2,1}^{(0)}$ element. With the usual notations, you can choose:

\begin{equation*}
\forall k\in \N, \ k \se 2, \ \ E_{2k,1}^{(0)} = E_{2k}
\end{equation*}
because Eisenstein's series are well known and admit various representations.
%L'absence d'un élément $E_{2,1}^{(0)} \not=0$ ne permettant pas de choisir $E_{2k,1}^{(0)} = [E_{2,1}^{(0)}]^k$.

A more economical choice, and in the spirit of Theorem II-\ref{ThmStruct101}, is to ask
\begin{equation*}
\left\{
\begin{array}{lcll}
\ds E_{2k,1}^{(0)} & = & E_4^{k/2} & \text{si} \ k\in 2\N^*\\
\ds E_{2k,1}^{(0)} & = & E_4^{(k-3)/2} E_6 & \text{si} \ k\in 2\N^*+1.
\end{array}
\right.
\end{equation*}
\ms

Anyway, for $k = 6q+r\se 2$, $1\ie r \ie 6$, a $M_{2k}(\Gamma_0(1))$ basis is given by
\begin{equation*}
\left\{
\begin{array}{ll}
\ds (\Delta^n E_{2k-12n,1}^{(0)})_{0\ie n\ie q-1} & \text{si} \ r=1\\
\ds (\Delta^n E_{2k-12n,1}^{(0)})_{0\ie n\ie q} & \text{si} \ 2\ie r \ie 5\\
\ds (\Delta^{q+1})\cup(\Delta^n E_{2k-12n,1}^{(0)})_{0\ie n\ie q} & \text{si} \ r=6.
\end{array}
\right.
\end{equation*}
This is a unitary and upper triangular basis without jump.

In order to be consistent with the announced objectives, equality can be noted
\begin{equation}
\Delta(\tau) = \frac{1}{256} \pa{\twpa\pa{\frac{1}{2},\tau}\twpa\pa{\frac{\tau}{2},\tau}\twpa\pa{\frac{\tau+1}{2},\tau+1}}^2
\end{equation}

obtained, for example, with the $(\ref{Facttwpa})$ factorization of $\twpa$. We will find representations of $E_4$ and $E_6$ depending on $\twpa$ in paragraph 4.4.
\ms

Remember that the $\twpa$ function systematically leads to modular units (no cancellation on $\cal H$), and that this is also the case for products of such functions. Thus, any modular form according to $\Gamma_0(1) = SL_2(\Z)$ is a linear combination of modular units.
\bs
\ms

\subsection{ -- Structure and bases of $(M_{2k}(\Gamma_0(2)))_{k\in \N^*}$}
\ms

The strong modular unit that structures this set of spaces is:

\begin{equation}
\Delta_2(\tau) =  \eta(2\tau)^{16} \eta(\tau)^{-8} = q\prod _{k=1}^{+\infty}{\frac { \left( 1-{q}^{2k} \right) ^{16}}{ \left(1-{q}^{k} \right) ^{8}}} \in M_{4}(\Gamma_0(2)).
\end{equation}
From the ${\twpa}$ factorization, we deduct:
\begin{equation}
\Delta_2(\tau) = \frac{1}{256} {\twpa}(\frac{1}{2},\tau)^2.
\end{equation}

Before we continue, let us note one important point. The zeros of $z\mapsto \twpa(z,\tau)$ are on the lattice $\Lambda_\tau$ and $\frac{1}{2}$ is never on this lattice. This remark spreads easily, and explains why $\twpa$ is a good tool to obtain modular units (not necessarily strong because one does not control the behavior at the cusps).
\ms

Let us remind the first values of $d_{2k}(2)$:

\begin{center}
\begin{tabular}{c|c|c|c|c|c|c|c|c}
$2k$   & 2  & 4  & 6  & 8  & 10  & 12  & 14  & 16\\
\hline
$d_{2k}(2)$   & 1  & 2  & 2  & 3  & 3  & 4  & 4  & 5 \\
\end{tabular}
\end{center}

\bi
\item[$\bullet$] $M_{2}(\Gamma_0(2))$
\ms

We established the $(\ref{M22})$ relationship earlier, $\ds {\wpa}(\tau,2\tau) \in M_2(\Gamma_0(2))$, and we set:
\be
\item[$\star$]
$\ds E_{2,2}^{(0)}(\tau) = -3\wpa(\tau,2\tau) = 1+24q+24q^2+96q^3+24q^4 + O(q^5)$.
\ee
\ms

Other representations are possible. For example based on $\twpa$:
\be
%\item[$\star$]
%$\ds E_{2,2}^{(0)}(\tau) = {\twpa}(\tau+\frac{1}{2},2\tau+1) - {\twpa}(\frac{1}{2},2\tau)$

\item[$\star$]
$\ds E_{2,2}^{(0)}(\tau) = {\twpa}(\tau,2\tau) - 2{\twpa}(\frac{1}{2},2\tau)$.

\ee

This equality will be demonstrated by modular argument when studying space $M_{2}(\Gamma_0(4))$.
%Comme cela a été fait pour $\Delta_2$ ou $E_{2,2}^{(1)}$, on peut démontrer directement l wiappartenance de ces fonctions à $M_{2}(\Gamma_0(2))$ avant de conclure. Une méthode, incorrect du point de vue mathématique mais beaucoup plus efficace, consiste à obtenir un développement en séries de Fourier d wiune fonction à tester (ce qui est facile ici, compte tenu des formules ?? et ??) et d wiidentifier cette fonction grâce à une base donnée par SAGE.
%\ms

%Afin de ne pas alourdir l wiarticle [Feau???], toutes les démonstrations de modularités des fonctions à venir n wiont pas été reportées, tout en sachant qu wielles sont possibles. Nous invitons néanmoins le lecteur à vérifier les indices probants de modularité grâce aux bases fournies par SAGE.
\bs

\item[$\bullet$] $M_{4}(\Gamma_0(2))$
\ms

We choose $E_{4,2}^{(0)} =[E_{2,2}^{(0)}]^2$, and since $\Delta_2\in M_{4}(\Gamma_0(2))$, we have a unitary upper triangular basis:

\be
\item[$\star$]
$\ds E_{4,2}^{(0)}(\tau) = [E_{2,2}^{(0)}(\tau)]^2 = 9 {\wpa}(\tau,2\tau)^2 = 1+48q+624q^2+1344q^3+5232q^4 + O(q^5)$.

\item[$\star$]
$\ds E_{4,2}^{(1)}(\tau) = \Delta_2(\tau) = \frac{1}{256} {\twpa}(\frac{1}{2},\tau)^2 = q+8q^2+28q^3+64q^4+ O(q^{5})$
\ee
\bs

\item[$\bullet$] The general $M_{2k}(\Gamma_0(2))$ case
\ms

We can apply Theorem II- \ref{ThmStruct101} with $\nu(\Delta_2) = 1$:
\[\forall k\se 3, \ \ M_{2k}(\Gamma_0(2)) = Vect(E_{2k,2}^{(0)})\oplus \Delta_2 . M_{2k-4}(\Gamma_0(2))\]

which is consistent with the dimension table.\ms

We choose $E_{2k,2}^{(0)} =[E_{2,2}^{(0)}]^{k}$, and we deduct a $M_{2k}(\Gamma_0(2))$ unitary upper triangular basis:
\begin{equation}
{\cal B}_{2k}(\Gamma_0(2)) = \pa{[E_{2,2}^{(0)}]^a \Delta_2^b, \ {\rm with} \ (a,b)\in\N^2 \ {\rm such \ that} \ a+2b = k}.
\end{equation}
\bs

Finally, we see that the products of the two modular forms ${\twpa}(\tau,2\tau) - 2{\twpa}(\frac{1}{2},2\tau)$ and ${\twpa}(\frac{1}{2},\tau)^2$ generate, by linear combinations of homogeneous weights, all modular forms according to $\Gamma_0(2)$\ldots \ And therefore also according to $\Gamma_0(1) = SL_2(\Z)$.

We could check, for example, the equalities for the Eisenstein series $E_4$:

\begin{equation}
\begin{array}{lcl}
E_4(\tau) & = & \ds \frac{1}{2} \pa{{\twpa}(\frac{1}{2},\tau)^2 + {\twpa}(\frac{\tau}{2},\tau)^2 + {\twpa}(\frac{\tau+1}{2},\tau+1)^2}\\
         & = & \ds 3 \pa{{\wpa}(\frac{1}{2},\tau)^2 + {\wpa}(\frac{\tau}{2},\tau)^2 + {\wpa}(\frac{1}{2},\tau) {\wpa}(\frac{\tau}{2},\tau)}\\
         &  = & \ds 1+240q+2160q^2+6720q^3+17520q^4 + O(q^{5}).\label{Ei4}
\end{array}
\end{equation}
\ei
\bs
\ms

\bs

\subsection{-- Structure and bases of $(M_{2k}(\Gamma_0(3)))_{k\in \N^*}$}
\ms

The strong modular unit $\Delta_3$ that structures this set of spaces is defined as follows:
\[\Delta_3(\tau) = \eta(3\tau)^{18} \eta(\tau)^{-6} = {q}^{2}\prod _{k=1}^{+\infty} \frac{(1-{q}^{3k})^{18}}{(1-{q}^{k})^{6}} \in M_{6}(\Gamma_0(3)).\]
\ms

Let us remind the first values of $d_{2k}(3)$:

\begin{center}
\begin{tabular}{ c|c|c|c|c|c|c|c|c}
$2k$   & 2  & 4  & 6  & 8  & 10  & 12  & 14  & 16 \\
\hline
$d_{2k}(3)$ & 1  & 2  & 3  & 3  & 4   &  5  & 5   & 6  \\
\end{tabular}
\end{center}

\bi
\item[$\bullet$] $M_{2}(\Gamma_0(3))$
\ms

We have several versions of the $M_2(\Gamma_0(3))$ unitary generator:

\be
\item[$\star$]
The element $E_{2,3}^{(0)}$:

\begin{tabular}{lcl}
$E_{2,3}^{(0)}(\tau) $ & = & $-3{\wpa}(\tau,3\tau) = \ds -3{\wpa}(2\tau,3\tau)$\\
%                & = & $\ds -{\wpa}(2\tau,3\tau)$\\
                & = & $\ds -3 {\twpa}(\frac{\tau+1}{2},3\tau) + {\twpa}(\frac{1}{2},3\tau) + {\twpa}(\frac{3\tau}{2},3\tau)$\\
                & = & $1+12q+36q^2+12q^3+84q^4+O(q^5)$
\end{tabular}
\ee

\item[$\bullet$] $M_{4}(\Gamma_0(3))$
\ms

We can choose $\ds E_{4,3}^{(0)} = [E_{2,3}^{(0)}]^2\in M_{4}(\Gamma_0(3))$.

\be
\item[$\star$]
The element $E_{4,3}^{(0)}$:

\begin{tabular}{lcl}
$E_{4,3}^{(0)}(\tau)$ & = & $\ds 9{\wpa}(\tau,3\tau)^2$\\
         &  = & $\ds 1+24q+216q^2+888q^3+1752q^4 + O(q^5)$
\end{tabular}
\ms

In addition, the Eisenstein series $E_4(3\tau)$ is also a unitary element of $M_{4}(\Gamma_0(3))$, with null valuation and not colinear at $E_{4,3}^{(0)}$. By subtraction,

\item[$\star$]
The element $E_{4,3}^{(1)}$:

\begin{tabular}{lcl}
$E_{4,3}^{(1)}(\tau)$ & = & $\ds \frac{1}{8} \pa{3{\wpa}(\tau,3\tau)^2 -  {\wpa}(\frac{1}{2},3\tau)^2 - {\wpa}(\frac{3\tau}{2},3\tau)^2 - {\wpa}(\frac{1}{2},3\tau) {\wpa}(\frac{3\tau}{2},3\tau)}$\\
         &  = & $\ds q+9q^2+27q^3+73q^4+ O(q^5).$
%$E_{4,3}^{(1)}(\tau)$ & = & $\ds \frac{1}{48} \pa{18{\wpa}(\tau,3\tau)^2 -  {\twpa}(\frac{1}{2},3\tau)^2 - {\twpa}(\frac{3\tau}{2},3\tau)^2 - {\twpa}(\frac{3\tau+1}{2},3\tau+1)^2}$\\
%         &  = & $\ds q+9q^2+27q^3+73q^4+ O(q^5).$
\end{tabular}
\ee

\item[$\bullet$] $M_{6}(\Gamma_0(3))$
\ms

It comes naturally:
\be
\item[$\star$]
The element $E_{6,3}^{(0)}$:

\begin{tabular}{lcl}
$E_{6,3}^{(0)}(\tau)$ & = & $\ds -27 {\wpa}(\tau,3\tau)^3 = [E_{2,3}^{(0)}]^3$\\
         &  = & $\ds 1+36q+540q^2+4356q^3+20556q^4 + O(q^5).$
\end{tabular}

\item[$\star$]
The element $E_{6,3}^{(1)}$:

\begin{tabular}{lcl}
$E_{6,3}^{(1)}(\tau)$ & = & $E_{2,3}^{(0)}.E_{4,3}^{(1)}$\\
& = & $\ds -\frac{3}{8}{\wpa}(\tau,3\tau) \pa{3{\wpa}(\tau,3\tau)^2 -  {\wpa}(\frac{1}{2},3\tau)^2 - {\wpa}(\frac{3\tau}{2},3\tau)^2 - {\wpa}(\frac{1}{2},3\tau) {\wpa}(\frac{3\tau}{2},3\tau)}$\\
         &  = & $\ds q+21q^2+171q^3+733q^4+2166q^5+5535q^6+ O(q^7).$
%$E_{6,3}^{(1)}(\tau)$ & = & $E_{2,3}^{(0)}.E_{4,3}^{(1)}$\\
%& = & $\ds -\frac{{\wpa}(\tau,3\tau)}{16} \pa{18{\wpa}(\tau,3\tau)^2 -  {\twpa}(\frac{1}{2},3\tau)^2 - {\twpa}(\frac{3\tau}{2},3\tau)^2 - {\twpa}(\frac{3\tau+1}{2},3\tau+1)^2}$\\
%         &  = & $\ds q+21q^2+171q^3+733q^4+2166q^5+5535q^6+ O(q^7).$
\end{tabular}

\item[$\star$]
The element $E_{6,3}^{(2)}$:

\begin{tabular}{lcl}
$E_{6,3}^{(2)}(\tau)$ & = & $\Delta_3(\tau)$\\
         &  = & $\ds  {q}^{2}+6{q}^{3}+27{q}^{4}+80{q}^{5}+207{q}^{6} + O(q^7).$
\end{tabular}
\ee

\item[$\bullet$] The general $M_{2k}(\Gamma_0(3))$ case
\ms

For $k \se 4$ we choose the first two elements of a $M_{2k}(\Gamma_0(3))$ unitary upper triangular basis as follows:
\be
\item[$\star$]
The element $E_{2k,3}^{(0)}$:

\begin{tabular}{lcl}
$E_{2k,3}^{(0)}$ & = & $[E_{2,3}^{(0)}]^k$.
\end{tabular}

\item[$\star$]
The element $E_{2k,3}^{(1)}$:

\begin{tabular}{lcl}
$E_{2k,3}^{(1)}$ & = & $[E_{2,3}^{(0)}]^{k-2}.E_{4,3}^{(1)}$
\end{tabular}
\ee
\bs

According to Theorem II-\ref{ThmStruct101}:
\begin{equation*}
\forall k \se 4, \ \ M_{2k}(\Gamma_0(3)) = Vect(E_{2k,3}^{(0)},E_{2k,3}^{(1)}) \oplus \Delta_3.M_{2k-6}(\Gamma_0(3))
\end{equation*}

and for $k\se 4$, we deduct a $M_{2k}(\Gamma_0(3))$ basis:

\begin{equation*}
{\cal B}_{2k}(\Gamma_0(3)) = \pa{[E_{2,3}^{(0)}]^a.\Delta_3^b, \ (a,b)\in\N^2 \ / \ a+3b = k} \cup \pa{E_{4,3}^{(1)}.[E_{2,3}^{(0)}]^a.\Delta_3^b, \ (a,b)\in\N^2 \ / \ a+3b = k-2}.
\end{equation*}

\ei
\bs
\ms

\subsection{-- Structure and bases of $(M_{2k}(\Gamma_0(4)))_{k\in \N^*}$}
\ms

The modular form that structures this set of spaces is $\Delta_4$ defined as follows:

\[\Delta_4(\tau) = \eta(4\tau)^{8} \eta(2\tau)^{-4} = -\frac{1}{16} {\twpa}(\frac{1}{2},2\tau) = q\prod_{k=1}^{+\infty} \frac{(1-{q}^{4n})^{8}}{(1-{q}^{2n})^{4}} \in M_{2}(\Gamma_0(4)).\]

\bs

This time $\Delta_4$ belongs to $M_{2}(\Gamma_0(4))$, which gives a very simple structure for the level $4$ spaces of which here is the first dimensions table.
\ms

\begin{center}
\begin{tabular}{ c|c|c|c|c|c|c|c|c}
$2k$   & 2  & 4  & 6  & 8  & 10  & 12  & 14  & 16 \\
\hline
$d_{2k}(4)$    & 2  & 3  & 4  & 5  & 6  & 7  & 8  & 9 \\
\end{tabular}
\end{center}

\bi
\item[$\bullet$] $M_{2}(\Gamma_0(4))$
\ms

There is a lot of choices for $E_{2,4}^{(0)}$: $E_{2,2}^{(0)}$, ${\twpa}(\tau,2\tau)$, ${\wpa}(\tau,4\tau)$, the Eisentein $E_4$ series and others\ldots

This time we choose ${\twpa}(\tau,2\tau)$ for consistency with $\Delta_4$. Of course, we must have $E_{2,4}^{(1)} = \Delta_4$.
\be
\item[$\star$]
The element $E_{2,4}^{(0)}$:

\begin{tabular}{lcl}
$E_{2,4}^{(0)}(\tau)$ & = & $\ds {\twpa}(\tau,2\tau)$\\
								& = & $\ds \prod_{k=1}^{+\infty} \frac{(1-{q}^{n})^{8}}{(1-{q}^{2n})^{4}}$\\
                & = & $1-8q+24q^2+32q^3+24q^4+O(q^5).$
\end{tabular}

\item[$\star$]
The element $E_{2,4}^{(1)}$:

\begin{tabular}{lcl}
$E_{2,4}^{(1)}(\tau)$ & = & $\ds -\frac{1}{16} {\twpa}(\frac{1}{2},2\tau)$\\
  & = & $\ds q\prod_{k=1}^{+\infty} \frac{(1-{q}^{4n})^{8}}{(1-{q}^{2n})^{4}}$\\
                & = & $\ds q+4q^3+6q^5+8q^7+13q^9+O(q^{11}).$
\end{tabular}
\ee

With these choices, we have a basis of two modular units for $M_{2}(\Gamma_0(4))$.
\ms

In application, let us verify the equality announced during the search for a $M_2(\Gamma_0(2))$ basis:
\[{\wpa}(\tau,2\tau) = {\twpa}(\tau,2\tau) - 2{\twpa}(\frac{1}{2},2\tau).\]

We know that $\ds E_{2,2}^{(0)}(\tau) = {\wpa}(\tau,2\tau) = 1+24q+O(q^3) \in M_{2}(\Gamma_0(2)) \subset M_{2}(\Gamma_0(4))$.

Family $({\twpa}(\tau,2\tau), -\frac{1}{16} {\twpa}(\frac{1}{2},2\tau))$ is a unitary upper triangular $M_{2}(\Gamma_0(4))$ basis, and therefore
\[{\wpa}(\tau,2\tau) = E_{2,4}^{(0)} + 32E_{2,4}^{(1)} = {\twpa}(\tau,2\tau) - 2{\twpa}(\frac{1}{2},2\tau)).\]
\ms

\item[$\bullet$] The general $M_{2k}(\Gamma_0(4))$ case
\ms

For $k \se 2$, we choose the first element of a $M_{2k}(\Gamma_0(4))$ unitary upper triangular basis as follows:

\be
\item[$\star$]
The element $E_{2k,4}^{(0)}$:

\begin{tabular}{lcl}
$E_{2k,4}^{(0)}$ & = & $[E_{2,4}^{(0)}]^k.$
\end{tabular}
\ee

We deduce:
\[\forall k \se 2, \ \ M_{2k}(\Gamma_0(4)) = Vect(E_{2k,4}^{(0)}) \oplus \Delta_4.M_{2k-2}(\Gamma_0(4)).\]

Then we get a $M_{2k}(\Gamma_0(4))$ basis:

\[{\cal B}_{2k}(\Gamma_0(4)) = \pa{[E_{2,4}^{(0)}]^a.\Delta_4^b, \ {\rm with} \ (a,b)\in\N^2 \ {\rm such \ that} \ a+b = k}.\]
\ei
\bs

The level $4$ and weight $2k$ modular forms are exactly the $k$ degree homogeneous polynomials in the variables $\twpa(\tau,2\tau)$ and ${\twpa}(\frac{1}{2},2\tau)$. Consequently, Eisenstein series are written in this form, as well as all modular forms according to $\Gamma_0(1)$. Also, since these two modular forms are $\eta$-products, we find that the level $1$ modular forms are linear combinations of $\eta$-products.
\ms

When we compare the early terms of the asymptotic expansions, we find representations other than $(\ref{Ei4})$:
\begin{equation*}
\begin{array}{lcl}
E_4(\tau) & = & \ds \twpa(\tau,2\tau)^2 + 16 \twpa(\frac{1}{2},2\tau)^2 - 16 \twpa(\tau,2\tau) \twpa(\frac{1}{2},2\tau)\\
 & = & \ds \twpa(\frac{\tau}{2},\tau)^2 + \twpa(\frac{1}{2},\tau)^2 - \twpa(\frac{\tau}{2},\tau)\twpa(\frac{1}{2},\tau).
\end{array}
\end{equation*}
The first equality was expected, but not the second which is symmetrical and refers to the period $\tau$ and not $2\tau$. 

Likewise:
\begin{equation*}
\begin{array}{lcl}
E_6(\tau) & = & \ds \twpa(\tau,2\tau)^3 + 30 \twpa(\tau,2\tau)^2 \twpa(\frac{1}{2},2\tau) - 96 \twpa(\tau,2\tau)\twpa(\frac{1}{2},2\tau)^2 +64\twpa(\frac{1}{2},2\tau)^3\\
& = & \ds \twpa(\frac{\tau}{2},\tau)^3 - \frac{3}{2} \twpa(\frac{\tau}{2},\tau)^2 \twpa(\frac{1}{2},\tau) - \frac{3}{2} \twpa(\frac{\tau}{2},\tau)\twpa(\frac{1}{2},\tau)^2 + \twpa(\frac{1}{2},\tau)^3.
\end{array}
\end{equation*}
\ms

For $E_8$, $E_{10}$ and $E_{12}$:
\begin{equation*}
\begin{array}{lcl}
E_8(\tau) & = & \ds \twpa(\tau,2\tau)^4 - 32\twpa(\tau,2\tau)^3\twpa(\frac{1}{2},2\tau) + 288 \twpa(\tau,2\tau)^2\twpa(\frac{1}{2},2\tau)^2 -512 \twpa(\tau,2\tau) \twpa(\frac{1}{2},2\tau)^3  + 256\twpa(\frac{1}{2},2\tau)^4\\
& = & \ds \twpa(\frac{\tau}{2},\tau)^4 - 2 \twpa(\frac{\tau}{2},\tau)^3 \twpa(\frac{1}{2},\tau) + 3 \twpa(\frac{\tau}{2},\tau)^2 \twpa(\frac{1}{2},\tau)^2 -2 \twpa(\frac{\tau}{2},\tau) \twpa(\frac{1}{2},\tau)^3 + \twpa(\frac{1}{2},\tau)^4.
\end{array}
\end{equation*}

\begin{equation*}
\begin{array}{lcl}
E_{10}(\tau) & = & \ds \twpa(\frac{\tau}{2},\tau)^5 - \frac{5}{2} \twpa(\frac{\tau}{2},\tau)^4 \twpa(\frac{1}{2},\tau) +  \twpa(\frac{\tau}{2},\tau)^3 \twpa(\frac{1}{2},\tau)^2 + \twpa(\frac{\tau}{2},\tau)^2 \twpa(\frac{1}{2},\tau)^3 \vspace{3mm}\\
& &  \hspace{7cm}  \ds - \frac{5}{2} \twpa(\frac{\tau}{2},\tau) \twpa(\frac{1}{2},\tau)^4 + \twpa(\frac{1}{2},\tau)^5.
\end{array}
\end{equation*}

\begin{equation*}
\begin{array}{lcl}
E_{12}(\tau) & = & \ds \twpa(\frac{\tau}{2},\tau)^6 - 3 \twpa(\frac{\tau}{2},\tau)^5 \twpa(\frac{1}{2},\tau) + \frac{4917}{1382} \twpa(\frac{\tau}{2},\tau)^4 \twpa(\frac{1}{2},\tau)^2 - \frac{1462}{691} \twpa(\frac{\tau}{2},\tau)^3 \twpa(\frac{1}{2},\tau)^3 +\vspace{3mm}\\

& & \ds \hspace{4.5cm} \frac{4917}{1382} \twpa(\frac{\tau}{2},\tau)^2 \twpa(\frac{1}{2},\tau)^4 -3 \twpa(\frac{\tau}{2},\tau) \twpa(\frac{1}{2},\tau)^5 + \twpa(\frac{1}{2},\tau)^6.
\end{array}
\end{equation*}

\bs
\ms

\subsection{-- Structure and bases of $(M_{2k}(\Gamma_0(5)))_{k\in \N^*}$}
\ms

When $N = 5$, the strong modular unit that structures the modular spaces is:

\[\Delta_5(\tau) = \eta(5\tau)^{10}\eta(\tau)^{-2} = \frac{1}{16} ({\wpa}(\tau,5\tau) - {\wpa}(2\tau,5\tau))^2 = {q}^{2}\prod _{n=1}^{+\infty}{\frac { \left( 1-{q}^{5n}
 \right) ^{10}}{ \left( 1-{q}^{n} \right) ^{2}}}\in M_{4}(\Gamma_0(5)).\]

\bs

This time $\Delta_5$ belongs to $M_{4}(\Gamma_0(5))$, so it is necessary to explain $M_{2}(\Gamma_0(5))$ and $M_{4}(\Gamma_0(5))$ bases.
\ms

The table of the dimensions of the first spaces:

\begin{center}
\begin{tabular}{ c|c|c|c|c|c|c|c|c}
$2k$   & 2  & 4  & 6  & 8  & 10  & 12  & 14  & 16 \\
\hline
$d_{2k}(5)$    & 1  & 3  & 3  & 5  & 5  & 7  & 7  & 9 \\
\end{tabular}
\end{center}

\bi
\item[$\bullet$] $M_{2}(\Gamma_0(5))$
\ms

We know the standard $M_{2}(\Gamma_0(5))$ generator.

\be
\item[$\star$]
The element $E_{2,5}^{(0)}$:

\begin{tabular}{lcl}
$E_{2,5}^{(0)}(\tau)$ & = & $\ds -\frac{3}{4} \sum_{k=1}^4 {\wpa}(k\tau,5\tau)$\\
                & = & $\ds -\frac{3}{2} ({\wpa}(\tau,5\tau) + {\wpa}(2\tau,5\tau))$\\
                & = & $1+6q+18q^2+24q^3+42q^4+O(q^5)$
\end{tabular}
\ee
\ms

\item[$\bullet$] $M_{4}(\Gamma_0(5))$
\ms

A somewhat systematic study once again finds, in addition to $\Delta_5$, two elements to form a $M_{4}(\Gamma_0(5))$ basis.
\ms

\be
\item[(i)]
$\ds ({\wpa}(\tau,5\tau) + {\wpa}(2\tau,5\tau)^2 = \frac{4}{9} (1+12q+72q^2+264q^3 + O(q^4))$

\item[(ii)]
$\ds \frac{1}{16}({\wpa}(\tau,5\tau) - {\wpa}(2\tau,5\tau))^2 \ds  =  \Delta_5(\tau) = q^2+2q^3+5q^4+10q^5+20q^6+26q^7+45q^8+O(q^{9})$

\item[(iii)]
Generically, the Eisenstein series $E_4(5\tau)$:

$E_4(q^5) = \ds 3 \pa{{\wpa}(\frac{1}{2},5\tau)^2 + {\wpa}(\frac{5\tau}{2},5\tau)^2 + {\wpa}(\frac{1}{2},5\tau) {\wpa}(\frac{5\tau}{2},5\tau)} = 1 +240q^5+2160q^{10} + O(q^{11})$
\ee

These three functions are linearly independent and a combination of $(i)$ and $(ii)$ results in a $0$ and $1$ unit scale family that completes $\Delta_5$.

\be
\item[$\star$]
The element $E_{4,5}^{(0)}$:

\begin{tabular}{lcl}
$E_{4,5}^{(0)}(\tau)$ & = & $[E_{2,5}^{(0)}(\tau)]^2$\\
                & = & $\ds \frac{9}{4} \pa{{\wpa}(\tau,5\tau) + {\wpa}(2\tau,5\tau)}^2$\\
                & = & $1+12q+72{q}^{2}+264{q}^{3}+696{q}^{4}+O(q^{5}).
$
\end{tabular}
\ms

\item[$\star$]
The element $E_{4,5}^{(1)}$:

\begin{tabular}{lcl}
$E_{4,5}^{(1)}(\tau)$ & = & $\ds \frac{1}{48} \pa{9 \pa{{\wpa}(\tau,5\tau) + {\wpa}(2\tau,5\tau)}^2 -12 \pa{{\wpa}(\frac{1}{2},5\tau)^2 + {\wpa}(\frac{5\tau}{2},5\tau)^2 + {\wpa}(\frac{1}{2},5\tau) {\wpa}(\frac{5\tau}{2},5\tau)}}$\\
                & = & $q+6q^2+22q^3+58q^4+O(q^5)$
\end{tabular}
\ms

\item[$\star$]
The element $E_{4,5}^{(2)}$:

\begin{tabular}{lcl}
$E_{4,5}^{(2)}(\tau)$ & = & $\Delta_5(\tau)$\\
 & = & $\ds \frac{1}{16} ({\wpa}(\tau,5\tau) - {\wpa}(2\tau,5\tau))^2$\\
 & = & $\ds {q}^{2}\prod _{n=1}^{+\infty}{\frac { \left( 1-{q}^{5n}
 \right) ^{10}}{ \left( 1-{q}^{n} \right) ^{2}}}$\\
 & = & $\ds  q^2+2q^3+5q^4+O(q^{10})$\\
\end{tabular}
\ee
\ms

\item[$\bullet$] The general $M_{2k}(\Gamma_0(5))$ case
\ms

For $k \se 3$, you can choose the first two elements of a $M_{2k}(\Gamma_0(5))$ unitary upper triangular basis as follows:

\be
\item[$\star$]
The element $E_{2k,5}^{(0)}$:

\begin{tabular}{lcl}
$E_{2k,5}^{(0)}$ & = & $[E_{2,5}^{(0)}]^k$
\end{tabular}

\item[$\star$]
The element $E_{2k,5}^{(1)}$:

\begin{tabular}{lcl}
$E_{2k,5}^{(1)}$ & = & $[E_{2,5}^{(0)}]^{k-2}.E_{4,5}^{(1)}$
\end{tabular}
\ee

From $\Delta_5$ properties, we deduce by factorization arguments already developed:

\[\forall k \se 3, \ \ M_{2k}(\Gamma_0(5)) = Vect(E_{2k,5}^{(0)},E_{2k,5}^{(1)}) \oplus \Delta_5.M_{2k-4}(\Gamma_0(5))\]

to get a $M_{2k}(\Gamma_0(5))$ basis:

\[{\cal B}_{2k}(\Gamma_0(5)) = \pa{[E_{2,5}^{(0)}]^a.\Delta_5^b, \ (a,b)\in\N^2 \ / \ a+2b = k} \cup \pa{E_{4,5}^{(1)}.[E_{2,5}^{(0)}]^a.\Delta_5^b, \ (a,b)\in\N^2 \ / \ a+2b = k-2} .\]
\ei
\bs
\ms

\subsection{-- Structure and bases of $(M_{2k}(\Gamma_0(6)))_{k\in \N^*}$}
\ms
%On trouve dans Test $M_2(G_0(3))$ et $M_2(G_0(5))$ - 15 juillet 2015 la fonction suivante:

%\[\pa{{\pi^2\twpa}(\frac{1}{2},\tau){\pi^2\twpa}(\frac{\tau}{2},\tau){\pi^2\twpa}(\frac{\tau+1}{2},\tau+1)}^2 \in S_{12}(\Gamma_0(6).\]

We know that
\[\Delta_6(\tau) = \frac{\eta(\tau)^2 \eta(6\tau)^{12}}{\eta(2\tau)^4 \eta(3\tau)^6} = q^2\prod_{k=1}^{+\infty} \frac{(1-q^{k})^{2}(1-q^{6k})^{12}} {(1-q^{2k})^{4} (1-q^{3k})^{6}} \in M_{2}(\Gamma_0(6)).\]
\ms

The dimensions of the first spaces $(M_{2k}(\Gamma_0(6)))_{k\in \N^*}$ are shown in the following table:
\begin{center}
\begin{tabular}{ c|c|c|c|c|c|c|c|c}
$2k$   & 2  & 4  & 6  & 8  & 10  & 12  & 14  & 16 \\
\hline
$d_{2k}(6)$   & 3  & 5  & 7  & 9  &11  & 13  & 15  & 17 \\
\end{tabular}
\end{center}
\bs

\bi
\item[$\bullet$] $M_{2}(\Gamma_0(6))$
\ms

A now classic search for elements of $M_{2}(\Gamma_0(6))$ gives the forms:
\be
\item[(i)]
$\ds \sum_{k=1}^5 {\wpa}(k\tau,6\tau) = -\frac{1}{3}(5+24q+72q^2+96q^3+168q^4+O(q^5))$

%C'est une construction générique pour trouver un élement de $M_{2}(\Gamma_0(N))$.

\item[(ii)]
%$\ds {\twpa}(\tau,2\tau) - 2{\twpa}(\frac{1}{2},2\tau) = 1+24q+24q^2+96q^3+24q^4+144q^5+96q^6+192q^7+O(q^8)$

$\ds {\wpa}(\tau,2\tau) = -\frac{1}{3}(1+24q+24q^2+96q^3+24q^4+O(q^5))$

\item[(iii)]
$\ds {\wpa}(\tau,3\tau) = -\frac{1}{3}(1+12q+36q^2+12q^3+84q^4+O(q^5))$
\ee

These three modular forms are linearly independent, forming a $M_{2}(\Gamma_0(6))$ basis which can be reduced.
\be
\item[$\star$]
The element $E_{2,6}^{(0)}$:

\begin{tabular}{lcl}
$E_{2,6}^{(0)}(\tau)$ & = & $\ds -3{\wpa}(\tau,2\tau)$\\
                & = & $1+24q+24q^2+96q^3+24q^4+O(q^5)$
\end{tabular}

\item[$\star$]
The element $E_{2,6}^{(1)}$:

\begin{tabular}{lcl}
$E_{2,6}^{(1)}(\tau)$ & = & $\ds -\frac{1}{4}\pa{{\wpa}(\tau,2\tau) - {\wpa}(\tau,3\tau)}$\\
& = & $\ds q-q^2+7q^3-5q^4+ O(q^5)$
\end{tabular}

\item[$\star$]
The element $E_{2,6}^{(2)}$:

\begin{tabular}{lcl}
$E_{2,6}^{(2)}(\tau)$ & = & $\Delta_6(\tau)$\\
                &  = & $\ds \frac{1}{48} \pa{3\wpa(\tau,2\tau) - 8\wpa(\tau,3\tau) + \sum_{k=1}^5 {\wpa}(k\tau,6\tau)}$\\
                & = & $q^2-2q^3+3q^4+O(q^{5})$
\end{tabular}
\ee
\ms

\item[$\bullet$] The general $M_{2k}(\Gamma_0(6))$ case
\ms

The procedure is similar. For $k \se 2$, you can choose the first two elements of a $M_{2k}(\Gamma_0(6))$ unitary upper triangular basis as follows:

\be
\item[$\star$]
The element $E_{2k,6}^{(0)}$:

\begin{tabular}{lcl}
$E_{2k,6}^{(0)}(\tau)$ & = & $[E_{2,6}^{(0)}(\tau)]^k$
\end{tabular}

\item[$\star$]
The element $E_{2k,6}^{(1)}$:

\begin{tabular}{lcl}
$E_{2k,6}^{(1)}(\tau)$ & = & $[E_{2,6}^{(1)}(\tau)][E_{2,6}^{(0)}(\tau)]^{k-1}$
\end{tabular}
\ee

And as a result:
\[\forall k \se 2, \ \ M_{2k}(\Gamma_0(6)) = Vect(E_{2k,6}^{(0)},E_{2k,6}^{(1)}) \oplus \Delta_6.M_{2k-2}(\Gamma_0(4)).\]
Then we get a $M_{2k}(\Gamma_0(6))$ basis:

\[{\cal B}_{2k}(\Gamma_0(6)) = \pa{[E_{2,6}^{(0)}]^a.\Delta_6^b, \ (a,b)\in\N^2 \ / \ a+b = k} \cup \pa{E_{2,6}^{(1)}.[E_{2,6}^{(0)}]^a.\Delta_6^b, \ (a,b)\in\N^2 \ / \ a+b = k-1} .\]
\ei
\bs
\ms

\subsection{-- Structure and bases of $(M_{2k}(\Gamma_0(7)))_{k\in \N^*}$}
\ms

The strong modular unit of level $7$ is:
\[\Delta_7(\tau) = \eta(\tau)^{-2}\eta(7\tau)^{14} = q^4 \prod_{n=1}^{+\infty} \frac{(1-q^{7n})^{14}}{(1-q^n)^2} \in M_6(\Gamma_0(7)).\]

\ms

The dimensions of the first spaces $(M_{2k}(\Gamma_0(7)))_{k\in \N^*}$ are given in the following table:

\begin{center}
\begin{tabular}{ c|c|c|c|c|c|c|c|c}
$2k$   & 2  & 4  & 6  & 8  & 10  & 12  & 14  & 16 \\
\hline
$d_{2k}(7)$   & 1  & 3  & 5  & 5  & 7  & 9  & 9  & 11 \\
\end{tabular}
\end{center}
\bs

\bi
\item[$\bullet$] $M_{2}(\Gamma_0(7))$
\ms

A generator of $M_{2}(\Gamma_0(7))$ is given by the generic formula:
\be
\item[$\star$]
The element $E_{2,7}^{(0)}$:

\begin{tabular}{lcl}
$E_{2,7}^{(0)}(\tau)$ & = & $\ds - \sum_{k=1}^3 {\wpa}(k\tau,7\tau)$\\
                & = & $1+4q+12{q}^{2}+16{q}^{3}+28{q}^{4}+O (q^{5})$
\end{tabular}
\ee
\ms

\item[$\bullet$] $M_{4}(\Gamma_0(7))$

A search gives some elements of $M_{4}(\Gamma_0(7))$:

\be
\item[(i)]
$\ds \pa{\sum_{k=1}^3 {\wpa}(k\tau,7\tau)}^2 = 1+8q+40{q}^{2}+128{q}^{3}+328{q}^{4}+O ({q}^{5})$

\item[(ii)]
$\ds 3\sum_{k=1}^3 {\wpa}(k\tau,7\tau)^2 = 1+8q+72q^2+224q^3+584q^4+O(q^5)$

\item[(iii)]
$E_4(q^7) = \ds 3 \pa{{\wpa}(\frac{1}{2},7\tau)^2 + {\wpa}(\frac{7\tau}{2},7\tau)^2 + {\wpa}(\frac{1}{2},7\tau) {\wpa}(\frac{7\tau}{2},7\tau)} = 1 +240q^7+2160q^{14} + O(q^{21})$
\ee
\ms

A unitary upper triangular basis is obtained by the linear combinations $(i)$, $\frac{1}{8} ((i) - (iii))$, $\frac{1}{32}((ii)-(i))$

\be
\item[$\star$]
The element $E_{4,7}^{(0)}$:

\begin{tabular}{lcl}
$E_{4,7}^{(0)}(\tau)$ & = & $\ds [E_{2,7}^{(0)}(\tau)]^2$\\
                & = &  $\ds \pa{\sum_{k=1}^3 {\wpa}(k\tau,7\tau)}^2$\\
                & = & $ 1+8q+40{q}^{2}+128{q}^{3}+328{q}^{4}+656{q}^{5}+1216{q}^{
6}+1864{q}^{7}+O ({q}^{8})$
\end{tabular}

\item[$\star$]
The element $E_{4,7}^{(1)}$:

\begin{tabular}{lcl}
$E_{4,7}^{(1)}(\tau)$ & = & $\ds \frac{1}{8} \pa{\pa{\sum_{k=1}^3 {\wpa}(k\tau,7\tau)}^2 - 3 \pa{{\wpa}(\frac{1}{2},7\tau)^2 + {\wpa}(\frac{7\tau}{2},7\tau)^2 + {\wpa}(\frac{1}{2},7\tau) {\wpa}(\frac{7\tau}{2},7\tau)}}$\\
                & = & $ q+5q^2+16q^3+41q^4+82q^5+152q^6+203q^7+357q^8+O(q^{9})$
\end{tabular}

\item[$\star$]
The element $E_{4,7}^{(2)}$:

\begin{tabular}{lcl}
$E_{4,7}^{(2)}(\tau)$ & = & $\ds \frac{1}{32} \pa{3\sum_{k=1}^3 {\wpa}(k\tau,7\tau)^2 - \pa{\sum_{k=1}^3 {\wpa}(k\tau,7\tau)}^2}$\\
                & = & $ q^2 +3q^3+8q^4+11q^5+25q^6+35q^7+57q^8+O(q^9) $
\end{tabular}
\ee
\ms

\item[$\bullet$] $M_{6}(\Gamma_0(7))$
\ms

There is a begining for a unitary upper triangular basic: $([E_{2,7}^{(0)}]^3, E_{2,7}^{(0)} E_{4,7}^{(1)}, E_{2,7}^{(0)} E_{4,7}^{(2)})$. The function $\Delta_7$ is valuation $4$, there is a $3$ valuation item missing. To get it we can notice that $E_6$ is in $M_{6}(\Gamma_0(7))$ (just like $E_2(3\tau)$ by the way) and work by linear combinations.
\ms

We will rather introduce here a generic process outlined in $N = 4$ case: the homogeneous symmetrical polynomials of degree $3$ in $({\wpa}(k\tau,7\tau))_{1\ie k\ie 6}$, and by symmetry those of $({\wpa}(k\tau,7\tau))_{1\ie k\ie 3}$, are elements of $M_{6}(\Gamma_0(7))$.
\ms

\be
\item[(i)]
$\ds H_1(\tau) = 9 \pa{{\wpa}(\tau,7\tau)^3 + {\wpa}(2\tau,7\tau)^3 + {\wpa}(3\tau,7\tau)^3} = - (1+12q+180q^2+1200q^3+5124q^4+O(q^5))$

\item[(ii)]
$\ds H_2(\tau) = \frac{9}{2} \big({\wpa}(\tau,7\tau)^2{\wpa}(2\tau,7\tau) +{\wpa}(\tau,7\tau)^2{\wpa}(3\tau,7\tau) + {\wpa}(2\tau,7\tau)^2{\wpa}(\tau,7\tau) +{\wpa}(2\tau,7\tau)^2{\wpa}(3\tau,7\tau)$

{\hskip 1.0cm} $ +{\wpa}(3\tau,7\tau)^2{\wpa}(\tau,7\tau) +{\wpa}(3\tau,7\tau)^2{\wpa}(2\tau,7\tau)\big) = - (1+12q+84q^2+336q^3+1188q^4+O(q^5))$

\item[(iii)]
$\ds H_3(\tau) = 27 {\wpa}(\tau,7\tau){\wpa}(2\tau,7\tau){\wpa}(3\tau,7\tau) = - (1+12q+36q^2+192q^3+516q^4+O(q^5))$
\ee
\ms

It can be seen that

$\begin{array}{lcl}
\ds \frac{1}{576}(F_1(\tau)-3F_2(\tau)+2F_3(\tau)) & = & \ds -\frac{1}{128}[2{\wpa}(\tau,7\tau) - {\wpa}(2\tau,7\tau) - {\wpa}(3\tau,7\tau)]\times\\
&  & [2{\wpa}(2\tau,7\tau) - {\wpa}(\tau,7\tau) - {\wpa}(3\tau,7\tau)][2{\wpa}(3\tau,7\tau) - {\wpa}(\tau,7\tau) - {\wpa}(2\tau,7\tau)]\\
& = & q^3 + \frac{9}{2} q^4 + 12q^5 + O(q^6)
\end{array}$
which is our candidate for $E_{6,7}^{(3)}$.\ms

We can now describe a $M_{6}(\Gamma_0(7))$ unitary upper triangular basis.
\ms

\be
\item[$\star$]
The element $E_{6,7}^{(0)}$:

\begin{tabular}{lcl}
$E_{6,7}^{(0)}(\tau)$ & = & $\ds [E_{2,7}^{(0)}(\tau)]^3$\\
                & = & $1+12q+84q^2+400q^3+1476q^4+O(q^5)$
\end{tabular}
\ms

\item[$\star$]
The element $E_{6,7}^{(1)}$:

\begin{tabular}{lcl}
$E_{6,7}^{(1)}(\tau)$ & = & $\ds E_{2,7}^{(0)}(\tau)E_{4,7}^{(1)}(\tau)$\\
                & = & $q+9q^2+48q^3+181q^4+O(q^5)$
\end{tabular}
\ms

\item[$\star$]
The element $E_{6,7}^{(2)}$:

\begin{tabular}{lcl}
$E_{6,7}^{(2)}(\tau)$ & = & $\ds E_{2,7}^{(0)}(\tau)E_{4,7}^{(2)}(\tau)$\\
                & = & $q^2+7q^3+32q^4+O(q^5)$
\end{tabular}
\ms

\item[$\star$]
The element $E_{6,7}^{(3)}$:

\begin{tabular}{lcl}
$E_{6,7}^{(3)}(\tau)$ & = & $\ds \frac{1}{576} (H_1(\tau)-3H_2(\tau)+2H_3(\tau))$\\
                & = & $\ds q^3 + \frac{9}{2} q^4 + 12q^5 + O(q^6)$
\end{tabular}
\ms

\item[$\star$]
The element $E_{6,7}^{(4)}$:

\begin{tabular}{lcl}
$E_{6,7}^{(4)}(\tau)$ & = & $\ds \Delta_7(\tau)$\\
								& = & $\ds q^4 \prod_{n=1}^{+\infty} \frac{(1-q^{7n})^{14}}{(1-q^n)^2}$\\
                & = & $q^4+2q^5+5q^6+10q^7+O(q^{8})$
\end{tabular}
\ee

\item[$\bullet$] The general $M_{2k}(\Gamma_0(7))$ case
\ms

Based on structure Theorem II-\ref{ThmStruct101}:
\begin{equation}
\forall k \se 4, \ \ M_{2k}(\Gamma_0(7)) = Vect(E_{2k,7}^{(0)},E_{2k,7}^{(1)},E_{2k,7}^{(2)},E_{2k,7}^{(3)}) \oplus \Delta_7.M_{2k-6}(\Gamma_0(7)).\label{eqX}
\end{equation}

%La situation étant un peu plus complexe que précédemment, on met en place le procédé générique d'obtention des bases.

For $k \se 4$, we choose the first four elements of a $M_{2k}(\Gamma_0(7))$ unitary upper triangular basis as follows:

\be
\item[$\star$]
The element $E_{2k,7}^{(0)}$:

$E_{2k,7}^{(0)} = [E_{6,7}^{(0)}][E_{2,7}^{(0)}]^{k-3} = [E_{2,7}^{(0)}]^k$

\item[$\star$]
The element $E_{2k,7}^{(1)}$:

$E_{2k,7}^{(1)} = [E_{6,7}^{(1)}][E_{2,7}^{(0)}]^{k-3}$

\item[$\star$]
The element $E_{2k,7}^{(2)}$:

$E_{2k,7}^{(2)} = [E_{6,7}^{(2)}][E_{2,7}^{(0)}]^{k-3}$

\item[$\star$]
The element $E_{2k,7}^{(3)}$:

$E_{2k,7}^{(3)} = [E_{6,7}^{(3)}][E_{2,7}^{(0)}]^{k-3}$\ee
\ms

Equality $(\ref{eqX})$ enables one to recursively obtain a $M_{2k}(\Gamma_0(7))$ basis. For a literal description, we set
${\cal B}_2(\Gamma_0(7)) = (E_{2,7}^{(0)})$, ${\cal B}_4(\Gamma_0(7)) = (E_{4,7}^{(0)},E_{4,7}^{(1)},E_{4,7}^{(2)})$, ${\cal B}_6(\Gamma_0(7)) = (E_{6,7}^{(0)},E_{6,7}^{(1)},E_{6,7}^{(2)},E_{6,7}^{(3)})$.
\ms

For $k\se 4$, with $k = 3q + r$, $r\in \{1,2,3\}$:

\[{\cal B}_{2k}(\Gamma_0(7)) =  \Delta_7^q {\cal B}_{2r}(\Gamma_0(7)) \ \cup \ \bigcup_{r=0}^3 \pa{E_{6,7}^{(r)}.[E_{2,7}^{(0)}]^a.\Delta_7^b, \ (a,b)\in\N^2 \ / \ a+3b = k-3}.\]
\ei
\bs
\ms

\subsection{-- Structure and bases of $(M_{2k}(\Gamma_0(8)))_{k\in \N^*}$}
\ms

The dimensions of the first spaces $(M_{2k}(\Gamma_0(8)))_{k\in \N^*}$ are as follows:

\begin{center}
\begin{tabular}{ c|c|c|c|c|c|c|c|c}
$2k$   & 2  & 4  & 6  & 8  & 10  & 12  & 14  & 16 \\
\hline
$d_{2k}(8)$   & 3  & 5  & 7  & 9  & 11  & 13  & 15  & 17 \\
\end{tabular}
\end{center}

The structure here is very simple since $\Delta_8(\tau) = \Delta_4(2\tau)\in M_{2}(\Gamma_0(8))$.
\ms

\bi
\item[$\bullet$] $M_{2}(\Gamma_0(8))$
\ms

As in $N = 4$ case, it is possible to find a $M_{2}(\Gamma_0(8))$ basis consisting of modular units (strong for $\Delta_8$).
\ms

\be
\item[$\star$]
The element $E_{2,8}^{(0)}$:

\begin{tabular}{lcl}
$E_{2,8}^{(0)}(\tau)$ & = &  $\ds {\twpa}(\tau,2\tau)$\\
& = & $\ds \eta(\tau)^8\eta(2\tau)^{-4} = \prod_{k=1}^{+\infty} \frac{(1-{q}^{n})^{8}}{(1-{q}^{2n})^{4}}$\\
& = & $1-8q+24q^2-32q^3+24q^4+O(q^5)$.
\end{tabular}
\ms

\item[$\star$]
The element $E_{2,8}^{(1)}$:

\begin{tabular}{lcl}
$E_{2,8}^{(1)}(\tau)$ & = & $\ds -\frac{1}{16} {\twpa}(\frac{1}{2},2\tau)$\\
								& = & $\ds \Delta_4(\tau) = \eta(2\tau)^{-4} \eta(4\tau)^{8} = q\prod_{k=1}^{+\infty} \frac{(1-{q}^{4n})^{8}}{(1-{q}^{2n})^{4}}$\\
                & = & $\ds q+4q^3+6q^5+O(q^{6})$.
\end{tabular}

Note that $\Delta_4$ is not a strong modular unit of $(M_{2k}(\Gamma_0(8)))_{k\in \N^*}$. The function does not cancel in any non-infinite cusp according to $\Gamma_0(4)$ but such a cusp exists according to $\Gamma_0(8)$ !
\ms

\item[$\star$]
The element $E_{2,8}^{(2)}$:

\begin{tabular}{lcl}
$E_{2,8}^{(2)}(\tau)$ & = & $\ds -\frac{1}{16} {\twpa}(\frac{1}{2},4\tau)$\\
                & = & $ \ds \Delta_8(\tau) = \eta(4\tau)^{-4} \eta(8\tau)^{8} = q^2\prod_{k=1}^{+\infty} \frac{(1-{q}^{8n})^{8}}{(1-{q}^{4n})^{4}}$\\
                & = & $q^2+4q^6+O(q^{10})$.
\end{tabular}
\ee
\ms

\item[$\bullet$] The general $M_{2k}(\Gamma_0(8))$ case
\ms

We have equality:
\[\forall k \se 2, \ \ M_{2k}(\Gamma_0(8)) = Vect(E_{2k,8}^{(0)},E_{2k,8}^{(1)}) \oplus \Delta_8.M_{2k-2}(\Gamma_0(8)).\]

For $k \se 2$, we choose the first two elements of a $M_{2k}(\Gamma_0(8))$ unitary upper triangular basis as follows:

\be
\item[$\star$]
The element $E_{2k,8}^{(0)}$:

$\begin{array}{lcl}
E_{2k,8}^{(0)} & = & [E_{2,8}^{(0)}]^k
\end{array}$

\item[$\star$]
The element $E_{2k,8}^{(1)}$:

$\begin{array}{lcl}
E_{2k,8}^{(1)} & = & E_{2,8}^{(1)}[E_{2,8}^{(0)}]^{k-1}
\end{array}$
\ee
\ms

We finally get a $M_{2k}(\Gamma_0(8))$ basis:

\[{\cal B}_{2k}(\Gamma_0(8)) = \pa{[E_{2,8}^{(0)}]^a.\Delta_8^b, \ (a,b)\in\N^2 \ / \ a+b = k} \cup \pa{E_{2,8}^{(1)}.[E_{2,8}^{(0)}]^a.\Delta_8^b, \ (a,b)\in\N^2 \ / \ a+b = k-1}\]
\ei
\bs
\ms

\subsection{-- Structure and bases of $(M_{2k}(\Gamma_0(9)))_{k\in \N^*}$}
\ms

The modular form that structures this set of spaces is $\Delta_9$ defined by:

\[\Delta_9(\tau) = \eta(9\tau)^{6} \eta(3\tau)^{-2} = {q}^{2}\prod_{k=1}^{+\infty} \frac{(1-{q}^{9k})^{6}}{(1-{q}^{3k})^{2}} \in M_{2}(\Gamma_0(9)).\]

\bs

Let us remind the first values of $d_{2k}(9)$:

\begin{center}
\begin{tabular}{ c|c|c|c|c|c|c|c|c}
$2k$   & 2  & 4  & 6  & 8  & 10  & 12  & 14  & 16 \\
\hline
$d_{2k}(9)$ & 3  & 5  & 7  & 9  & 11   &  13  & 15   & 17  \\
\end{tabular}
\end{center}

\bi
\item[$\bullet$] $M_{2}(\Gamma_0(9))$
\ms

With the help of experience, we find the following elements:
\ms

\be
\item[(i)]
\begin{tabular}{lcl}
$\ds {\wpa}(\tau,3\tau)$ & = & $\ds -\frac{1}{3}(1+12q+36q^2+12q^3+84q^4+O(q^{5}))$\end{tabular}

\item[(ii)]
\begin{tabular}{lcl}
$\ds {\wpa}(3\tau,9\tau)$ & = & $\ds -\frac{1}{3}(1+12q^3+36q^6+O(q^{9}))$
\end{tabular}

\item[(iii)]
\begin{tabular}{lcl}
$\ds \sum_{k=1}^4 {\wpa}(k\tau,9\tau)$ & = & $\ds -\frac{4}{3}  (1+3q+9q^2+12q^3+21q^4+ O(q^{5})$
\end{tabular}

\item[(iv)]
\begin{tabular}{lcl}
$\ds \Delta_9(\tau)$ & = & $q^2+2q^5+5q^8 + O(q^{11})$
\end{tabular}
\ee
\ms

Note that the first three functions are bound in $M_{2}(\Gamma_0(9))$: $(i)+3(ii)-(iii) = 0$.
\ms

We obtain a unitary upper triangular basis as follows:
\ms

\be
\item[$\star$]
The element $E_{2,9}^{(0)}$:

\begin{tabular}{lcl}
$E_{2,9}^{(0)}(\tau)$ & = & $\ds -3 {\wpa}(3\tau,9\tau)$\\
                & = & $\ds 1+12q^3+36q^6+O(q^{9})$
\end{tabular}
\ms

\item[$\star$]
The element $E_{2,9}^{(1)}$:

\begin{tabular}{lcl}
$E_{2,9}^{(1)}(\tau)$ & = & $-\ds \frac{1}{4} \pa{{\wpa}(\tau,3\tau)-{\wpa}(3\tau,9\tau)}$\\
                & = & $\ds q+3q^2+7q^4+6q^5+O(q^{7}))$
\end{tabular}
\ms

\item[$\star$]
The element $E_{2,9}^{(2)}$:

\begin{tabular}{lcl}
$E_{2,9}^{(2)}(\tau)$ & = & $\Delta_9(\tau) = \ds \eta(9\tau)^{6} \eta(3\tau)^{-2}$\\
  & $=$ & $\ds {q}^{2}\prod_{k=1}^{+\infty}{\frac { \left( 1-{q}^{9k} \right) ^{6}}{ \left( 1-{q}^{3k} \right) ^{2}}}$\\
  & $=$ & $\ds q^2+2q^5+5q^8 + O(q^{11})$
\end{tabular}
\ee
\ms

\item[$\bullet$] The general $M_{2k}(\Gamma_0(9))$ case
\ms

The structure is isomorphic to $M_{2k}(\Gamma_0(8))$.
We have equality:
\[\forall k \se 2, \ \ M_{2k}(\Gamma_0(9)) = Vect(E_{2k,9}^{(0)},E_{2k,9}^{(1)}) \oplus \Delta_9.M_{2k-2}(\Gamma_0(9)).\]

For $k \se 2$, we choose the first two elements of a $M_{2k}(\Gamma_0(9))$ unitary upper triangular basis as follows:

\be
\item[$\star$]
The element $E_{2k,9}^{(0)}$:

$\begin{array}{lcl}
E_{2k,9}^{(0)} & = & [E_{2,9}^{(0)}]^k
\end{array}$

\item[$\star$]
The element $E_{2k,9}^{(1)}$:

$\begin{array}{lcl}
E_{2k,9}^{(1)} & = & E_{2,9}^{(1)}[E_{2,9}^{(0)}]^{k-1}
\end{array}$
\ee
\ms

Then we get a $M_{2k}(\Gamma_0(9))$ basis:

\[{\cal B}_{2k}(\Gamma_0(9)) = \pa{[E_{2,9}^{(0)}]^a.\Delta_9^b, \ (a,b)\in\N^2 \ / \ a+b = k} \cup \pa{E_{2,9}^{(1)}.[E_{2,9}^{(0)}]^a.\Delta_8^b, \ (a,b)\in\N^2 \ / \ a+b = k-1}\]
\ei
\bs
\ms

\subsection{-- Structure and bases of $(M_{2k}(\Gamma_0(10)))_{k\in \N^*}$}
\ms

The modular form that structures this set of spaces is $\Delta_{10}$ defined by:

\[\Delta_{10}(\tau) = \eta(\tau)^2 \eta(2\tau)^{-4} \eta(5\tau)^{-10} \eta(10\tau)^{20} = {q}^{6} \prod_{k=0}^{+\infty} \frac{( 1-q^{n})^2 (1-q^{10n})^{20}}{( 1-q^{2n})^4 ( 1-q^{5n})^{10}} \in M_{4}(\Gamma_0(10)).\]

\bs

Let us remind the first values of $d_{2k}(10)$:

\begin{center}
\begin{tabular}{ c|c|c|c|c|c|c|c|c}
$2k$   & 2  & 4  & 6  & 8  & 10  & 12  & 14  & 16 \\
\hline
$d_{2k}(10)$ & 3  & 7  & 9  & 13  & 15   &  19  & 21   & 25  \\
\end{tabular}
\end{center}

\bi
\item[$\bullet$] $M_{2}(\Gamma_0(10))$
\ms

From $M_{2}(\Gamma_0(2))$ and $M_{2}(\Gamma_0(5))$ bases, we find the following elements:
\ms

\be
\item[(i)]
\begin{tabular}{lcl}
$\ds {\wpa}(\tau,2\tau) $ & =  $\ds -\frac{1}{3} (1+24q+24q^2+96q^3+24q^4+144q^5+O(q^{6})$
\end{tabular}

\item[(ii)]
\begin{tabular}{lcl}
$\ds {\wpa}(5\tau,10\tau)$ & = & $\ds -\frac{1}{3}(1+24q^5+24q^{10}+O(q^{15}))$
\end{tabular}

\item[(iii)]
\begin{tabular}{lcl}
$\ds {\wpa}(\tau,5\tau) - {\wpa}(2\tau,5\tau)$ & = & $\ds -\frac{2}{3}  (1+6q+18q^2+24q^3+42q^4+6q^5+O(q^{6}))$
\end{tabular}

\item[(iv)]
\begin{tabular}{lcl}
$\ds \sum_{k=1}^4 {\wpa}(k\tau,10\tau)$ & = & $\ds -\frac{4}{3}  (1+3q+9q^2+12q^3+21q^4+15q^5+O(q^6))$
\end{tabular}
\ee
\ms

These functions are linked by the relationship $2(ii)+(iii)-(iv) = $0.

The following unitary upper triangular basis is derived.
\ms

\be
\item[$\star$]
The element $E_{2,10}^{(0)}$:

\begin{tabular}{lcl}
$E_{2,10}^{(0)}(\tau)$ & = & $-\ds 3 {\wpa}(5\tau,10\tau)$\\
                & = & $\ds 1+24q^5+24q^{10}+O(q^{15})$
\end{tabular}
\ms

\item[$\star$]
The element $E_{2,10}^{(1)}$:

\begin{tabular}{lcl}
$E_{2,10}^{(1)}(\tau)$ & = & $-\ds \frac{1}{8}\pa{{\wpa}(\tau,2\tau)-{\wpa}(5\tau,10\tau)}$\\
                 & = & $\ds q+q^2+4q^3+q^4+5q^5+4q^6+8q^7 + O(q^{8})$
\end{tabular}
\ms

\item[$\star$]
The element $E_{2,10}^{(2)}$:

\begin{tabular}{lcl}
$E_{2,10}^{(2)}(\tau)$ & = & $\ds \frac{1}{16}\pa{{\wpa}(\tau,2\tau)-2{\wpa}(\tau,5\tau)-2{\wpa}(2\tau,5\tau)+3{\wpa}(5\tau,10\tau)}$\\
                 & = & $\ds q^2+3q^4-4q^5+4q^6+O(q^{8})$
\end{tabular}
\ee
\ms

\item[$\bullet$] $M_{4}(\Gamma_0(10))$
\ms

The form $\Delta_{10}$ and the products of two elements of $M_{2}(\Gamma_0(10))$ provide elements for a $M_{4}(\Gamma_0(10))$ unitary upper triangular basis. Just missing a $5$ valuation element to complete the basis.

We know that the Eisenstein series $E_4(5\tau)$ is in $M_{4}(\Gamma_0(10))$, but it is a linear combination of the elements $(E_{4,10}^{(k)})_{0\ie k \ie 4}$ and therefore does not bring the element $E_{4,10}^{(5)}$ sought.

We can also think of $\Delta_5(2\tau)\in M_{4}(\Gamma_0(10))$. It is a $4$ valuation term but it is also a linear combination of $(E_{4,10}^{(k)})_{0\ie k \ie 4}$.
\ms

Finally, the last "natural" candidate who completes $\Delta_2(5\tau)$ basis, necessarily in $M_{4}(\Gamma_0(10))$ with valuation $5$.
\bs
\bs

\be
\item[$\star$]
The element $E_{4,10}^{(0)}$:

\begin{tabular}{lcl}
$E_{4,10}^{(0)}(\tau)$ & = & $[E_{2,10}^{(0)}(\tau)]^2 $\\
  & $=$ & $\ds 1+48\,{q}^{5}+624\,{q}^{10} + O(q^{15})$
\end{tabular}

\item[$\star$]
The element $E_{4,10}^{(1)}$:

\begin{tabular}{lcl}
$E_{4,10}^{(1)}(\tau)$ & = & $E_{2,10}^{(0)}(\tau) E_{2,10}^{(1)}(\tau) $\\
  & $=$ & $\ds q+{q}^{2}+4\,{q}^{3}+{q}^{4}+5\,{q}^{5}+28\,{q}^{6}+32\,{q}^{7} + O(q^{8})$
\end{tabular}

\item[$\star$]
The element $E_{4,10}^{(2)}$:

\begin{tabular}{lcl}
$E_{4,10}^{(2)}(\tau)$ & = & $ E_{2,10}^{(0)}(\tau)E_{2,10}^{(2)}(\tau) $\\
  & $=$ & $\ds {q}^{2}+3\,{q}^{4}-4\,{q}^{5}+4\,{q}^{6}+24\,{q}^{7}+ O(q^{8})$
\end{tabular}

\item[$\star$]
The element $E_{4,10}^{(3)}$:

\begin{tabular}{lcl}
$E_{4,10}^{(3)}(\tau)$ & = & $ E_{2,10}^{(1)}(\tau)E_{2,10}^{(2)}(\tau) $\\
  & $=$ & $\ds {q}^{3}+{q}^{4}+7\,{q}^{5}+17\,{q}^{7}+ O(q^{8})$
\end{tabular}

\item[$\star$]
The element $E_{4,10}^{(4)}$:

\begin{tabular}{lcl}
$E_{4,10}^{(4)}(\tau)$ & = & $ [E_{2,10}^{(2)}(\tau)]^2$\\
  & $=$ & $\ds {q}^{4}+6\,{q}^{6}-8\,{q}^{7}+ O(q^{8})$
\end{tabular}

\item[$\star$]
The element $E_{4,10}^{(5)}$:

\begin{tabular}{lcl}
$E_{4,10}^{(5)}(\tau)$ & = & $\Delta_2(5\tau) $\\
  & $=$ & $\ds \frac{1}{256} {\twpa}(\frac{1}{2},5\tau)^2$\\
  & $=$ & $\ds q^5+8q^{10}+28q^{15}+ O(q^{20})$
\end{tabular}

\item[$\star$]
The element $E_{4,10}^{(6)}$:

\begin{tabular}{lcl}
$E_{4,10}^{(6)}(\tau)$ & = & $\Delta_{10}(\tau)$\\
  & $=$ & $ \ds \eta(\tau)^2 \eta(2\tau)^{-4} \eta(5\tau)^{-10} \eta(10\tau)^{20}$\\
%  & $=$ & $\ds {q}^{6} \prod_{k=0}^{+\infty} \frac{( 1-q^{10k+1})^2 (1-{q}^{10k+3})^{2} (1-{q}^{10k+7})^{2} (1-{q}^{10k+9})^{2} (1-{q}^{10k+10})^{8}} {(1-{q}^{10k+2})^{2}(1-{q}^{10k+4})^{2} (1-{q}^{10k+5})^{8} (1-{q}^{10k+6})^{2} (1-{q}^{10k+8})^{2}}$\\
  & $=$ & $q^6-2q^7+3q^8-6q^9 + O(q^{10})$
\end{tabular}
\ee
\bs
\bs

\item[$\bullet$] The general $M_{2k}(\Gamma_0(10))$ case
\ms

We have the equality from Theorem II-\ref{ThmStruct101}:
\[\forall k \se 3, \ \ M_{2k}(\Gamma_0(10)) = Vect((E_{2k,10}^{(r)})_{0\ie r \ie 5}) \oplus \Delta_{10}.M_{2k-4}(\Gamma_0(10)).\]

%Comme pour le cas $N=7$, la structure des bases échelonnées étant plus complexe, on utilise un procédé générique.
For $k \se 3$, you can choose the first six elements of a $M_{2k}(\Gamma_0(10))$ unitary upper triangular basis as follows:

\be
\item[$\star$]
The element $E_{2k,10}^{(0)}$:

$E_{2k,10}^{(0)} = E_{4,10}^{(0)}][E_{2,10}^{(0)}]^{k-2} = [E_{2,10}^{(0)}]^k$

\item[$\star$]
The element $E_{2k,10}^{(1)}$:

$E_{2k,10}^{(1)} = [E_{4,10}^{(1)}][E_{2,10}^{(0)}]^{k-2}$

\item[$\star$]
The element $E_{2k,10}^{(2)}$:

$E_{2k,10}^{(1)} = [E_{4,10}^{(2)}][E_{2,10}^{(0)}]^{k-2}$

\item[$\star$]
The element $E_{2k,10}^{(3)}$:

$E_{2k,10}^{(1)} = [E_{4,10}^{(3)}][E_{2,10}^{(0)}]^{k-2}$

\item[$\star$]
The element $E_{2k,10}^{(4)}$:

$E_{2k,10}^{(1)} = [E_{4,10}^{(4)}][E_{2,10}^{(0)}]^{k-2}$

\item[$\star$]
The element $E_{2k,10}^{(5)}$:

$E_{2k,10}^{(5)} = [E_{4,10}^{(5)}][E_{2,10}^{(0)}]^{k-2}$
\ee
\ms

We then obtain a $M_{2k}(\Gamma_0(10))$ basis as follows.
\ms

We put ${\cal B}_2(\Gamma_0(10)) = (E_{2,10}{(r)})_{0\ie r \ie 2}$ and ${\cal B}_4(\Gamma_0(10)) = (E_{4,10}^{(r)})_{0\ie r \ie 6}$.
\ms

For $k\in \N^*$, by posing $k = 2q + r$, $r\in \{1,2\}$:

\[{\cal B}_{2k}(\Gamma_0(10)) =  \Delta_{10}^q {\cal B}_{2r}(\Gamma_0(10)) \ \cup \ \bigcup_{i=0}^5 \pa{E_{4,10}^{(i)}.[E_{2,10}^{(0)}]^a.\Delta_{10}^b, \ (a,b)\in\N^2 \ / \ a+2b = k-2}.\]
\ei
\bs
\ms

\section{-- Conclusion}
\ms

Essentially thanks to the introduction of the notion of strong modular unit, we were able to identify, in part I of this article \cite{FeauFM1}, the theoretical tools allowing one to structure the family of modular spaces $(M_{2k}(\Gamma_0(N)))_{k\in\N^*}$ and to obtain unitary upper triangular bases of them.
These results were applied to every values $1\ie N \ie 10$.
\ms

Although not all modularity proofs of the proposed functions are provided, this would produce a document as indigestible as it is voluminous, these have been done carefully for strong modular units in the first part of this article \cite{FeauFM1}. The modularity of other frequently appearing modular functions has also been verified, in particular thanks to the properties of the Weierstrass functions ${\wpa}$ and ${\twpa}$ studied in \cite{FeauE}.
\ms

These two elliptic functions were chosen ${\wpa}$ and ${\twpa}$ for several reasons.
\be
\item[(i)]
The ${\wpa}$ function is the elliptic Weierstrass function, while ${\twpa}$ is derived by renormalization. Both functions have an inversion symmetry that is almost immediately modular in character.
\item[(ii)]
Thanks to the lemmas \ref{Wp2Wpt} and \ref{Wpt2Wp}, one of the two functions ${\wpa}$ or ${\twpa}$ would have been enough to represent the bases of modular forms, but the formulations would have been heavier.
\item[(iii)]
For the modularity according to the congruence group $\Gamma_0(N)$, $1\ie N \ie 10$, and apart from the strong modular units represented in \cite{FeauFM1} in the form of $\eta$-products, these functions are sufficient to represent all the modular functions involved in this article.
\ee
\ms

This last point is important. The ${\wpa}$ and ${\twpa}$ functions naturally produce modular functions of weight $2$, then even weight per extension. If the congruence group enabled one access to odd degree modular spaces, for example $\Gamma_1(N)$, we would have chosen to represent the modular forms using Jacobi's elliptic functions, in particular $SD$, or even \cite{FeauE}. The functions ${\wpa}$ and ${\twpa}$ being themselves linked to Jacobi's functions, in particular $-\frac{1}{16}{\twpa}(z,\tau) = SD(z,\tau)^2$.
\ms

Other choices are possible to represent modular forms. In particular, for $k\in \N^*$, $\ds \frac{\partial^{2k} {\wpa}}{\partial^{2k} z}(z,\tau) = \frac{\partial^{2k} {\twpa}}{\partial^{2k} z}(z,\tau)$ naturally produces modular forms of weight $2k+2$.
\ms

We will have noted that the exceptional Eisenstein series $E_{2,2}^{(0)} M_{2}(\Gamma_0(2))$ appears here naturally since $\ds E_{2,2}^{(0)}(\tau) = -3{\wpa}(\tau,2\tau)$, and it is well known that, the Eisenstein series $E_{2k+2}$ is exactly $\ds \frac{\partial^{(2k)} {\wpa}}{\partial^{(2k)} z}(\tau,2\tau)$. The use of the derivation would have given greater clarity to the elements of the proposed bases, but the essential tools should have been expanded. Now the modular functions form a dense maquis in which it is better to have a good compass!
\ms

To close this second part of the article, I would like to acknowledge the exceptional work of the team in charge of the development of the SAGE module concerning modular functions. Many ideas have come from an interaction between theory and this powerful tool. And this, either to confirm an intuition or to guide towards the solution when the initial idea proved erroneous. Finally, this tool was a constant help to verify that a function was indeed in the expected modular space: a simple confrontation of the development in Fourier series with a unitary upper triangular basis of the modular space giving the answer in an almost certain way.
In return, this work provides high performance algorithms for the calculation of high weight modular forms.
\ms

For example, a modular form according to $\Gamma_0(10)$ with weight and valuation $2018$ is given by:
\[E_{2018,10}^{(2018)} = E_{4,10}^{(i)}[E_{2,10}^{(0)}]^a[\Delta_{10}]^b \ \text{avec} \ a+2b = 1007 \ \text{et} \ i+6b = 2018, \ 0\ie i \ie 5.\]
It comes $a = 335$, $b = 336$ and $i = 2$ for an almost immediate result on a personal computer:
\ms

$
\begin{array}{lcl}
E_{2018,10}^{(2018)} & = & {q}^{2018} - 672\,{q}^{2019}  +
226131\,{q}^{2020} - 50806116\,{q}^{2021} + 8574211132\,{q}^{2022} - \\
 &  & 1159385836896\,{q}^{2023} + 130843082948319\,{q}^{2024} -
12676560614152160\,{q}^{2025} +\\
 &  & + 1076314597159060977\,{q}^{2026} - 81359425707034726432\,{q}^{2027} + O(q^{2028})
\end{array}$

%%%%%%%%%%%%%%%%%%%%%%%%%%%%%%%%%%%%%%%%%%%%%%%%%%%%%%%%%%%%%%%%%%%%%%%%%%%%%%%%%%%%%%%%%%%%%%%%%%%%%%%%%%%%%%%%%%%%%%%%%%%%%%%%%%%%%%%%%%%%%%%%%%%%%%%%%%%%%%%%%%%%%%
%%%%%%%%%%%%%%%%%%%%%%%%%%%%%%%%%%%%%%%%%%%%%%%%%%%%%%%%%%%%%%%%%%%%%%%%%%%%%%%%%%%
\bs
\bs
\bs
\bs

\end{document}